\documentclass[12pt]{article}
\usepackage[final]{epsfig}
\usepackage{graphics}
\usepackage{amsmath}
\usepackage{amsfonts}
\usepackage{latexsym}
\usepackage{amssymb}
\usepackage{graphicx}
\usepackage{url}
\usepackage{epstopdf}

\usepackage{amsthm,color}

\newtheorem{lemma}{Lemma}[section]
\newtheorem{theorem}{Theorem}
\newtheorem{corollary}[lemma]{Corollary}

\theoremstyle{definition}
\newtheorem{remark}[lemma]{Remark}
\newtheorem{example}[lemma]{Example}
\newtheorem{definition}[lemma]{Definition}

\newcommand{\proofend}{$\Box$\bigskip}

\newcommand{\R}{{\mathbb R}}

\newcommand{\HH}{{\mathbb H}}
\newcommand{\Sph}{{\mathbb S}}
  
\newcommand{\Tr}{{\rm Tr\ }}

\def\Hess{\operatorname{Hess}}
\def\const{\mathrm{const}}

\def\grad{\mathrm{grad}\,}
\def\GL{\mathrm{GL}}
\def\diag{\operatorname{diag}}
\def\sign{\operatorname{sign}}

\begin{document}

\title{Ivory's Theorem revisited}

\author{Ivan Izmestiev\footnote{
Universit\'e de Fribourg,
D\'epartement de math\'ematiques,
Chemin du Mus\'ee 23,
CH-1700 Fribourg;
ivan.izmestiev@gmail.com}
\and
Serge Tabachnikov\footnote{
Department of Mathematics,
Pennsylvania State University,
University Park, PA 16802;
tabachni@math.psu.edu} 
}

\date{}
\maketitle

\begin{abstract}
Ivory's Lemma is a geometrical statement in the heart of J. Ivory's calculation of the gravitational potential of a homeoidal shell. In the simplest planar case, it claims that the diagonals of a curvilinear quadrilateral made by arcs of confocal ellipses and  hyperbolas are equal.

In the first part of this paper, we deduce Ivory's Lemma and its numerous generalizations from complete integrability of billiards on conics and quadrics. In the second part, we study analogs of Ivory's Lemma in Liouville and St\"ackel metrics. Our main focus is on the results of the German school of differential geometry obtained in the late 19 -- early 20th centuries that might be lesser know today.

In the third part, we generalize Newton's, Laplace's, and Ivory's theorems on gravitational and Coulomb potential of spheres and ellipsoids to the spherical and hyperbolic spaces. V. Arnold extended the results of Newton,  Laplace, and Ivory  to  algebraic hypersurfaces in Euclidean space; we generalize Arnold's theorem to the spaces of constant curvature. 
\end{abstract}

\newpage

\tableofcontents

\section{Introduction} \label{intro}

Theorems XXX and  XXXI of I. Newton's ``Principia" assert that {\it the gravitational field created by a spherical shell is zero in the region bounded by the shell, whereas, in the exterior region, the field is the same as the one created by the total mass of the shell concentrated at its center.}

P.-S. Laplace extended Newton's theorem to ellipsoids.
A \emph{homeoid} is the domain bounded by two homothetic ellipsoids with a common center.
Laplace proved the following theorem: {\it the gravitational field of a homogeneous homeoidal shell equals zero in the region bounded by the shell. If the shell is infinitely thin, then the equipotential surfaces in its exterior are the confocal ellipsoids.}

Laplace's proof was computational. J. Ivory's gave a different proof  \cite{Iv09} that used a geometric argument based on the lemma that now carries his name. 

Let $E_1$ and $E_2$ be origin-centered confocal ellipsoids in $\R^3$, and let $A$ be a linear map that takes $E_1$ to $E_2$. Let $P_1$ and $Q_1$ be points of $E_1$, and $P_2=A(P_1), Q_2=A(Q_2)$ be the corresponding points of $E_2$. The statement of Ivory's Lemma is as follows: $|P_1 Q_2|=|Q_1 P_2|$.

Ivory's Lemma is valid in all dimensions. In the simplest case of dimension two, the pairs of corresponding points lie on confocal hyperbolas, and the statement can be formulated as follows: {\it the diagonals of a quadrilateral made by arcs of confocal ellipses and hyperbolas are equal}, see Figure \ref{Ivst}.

\begin{figure}[hbtp]
\centering
\includegraphics[height=1.8in]{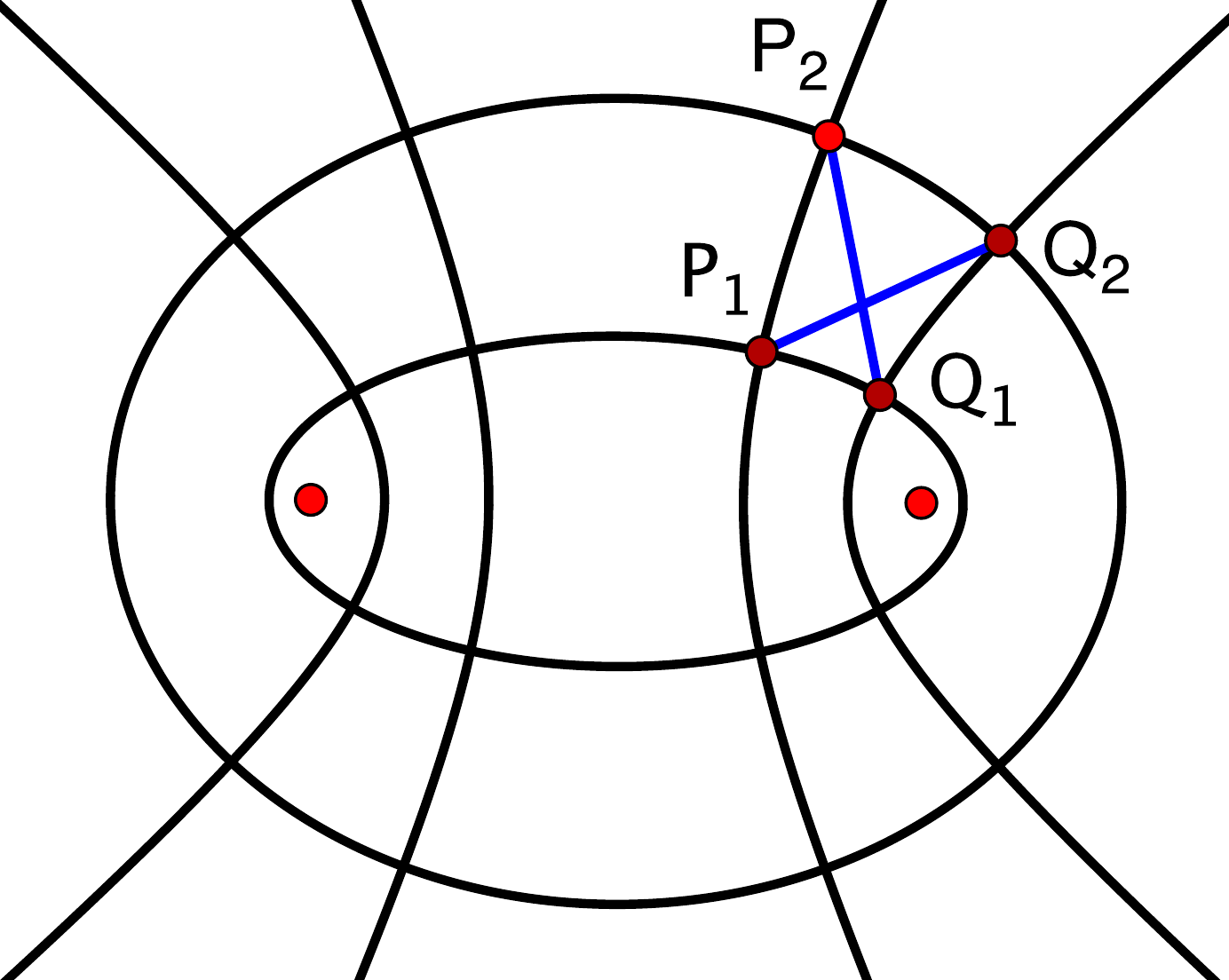}
\caption{Ivory's Lemma: $|P_1 Q_2|=|Q_1 P_2|$}
\label{Ivst}
\end{figure}

The theorems of Newton and of Laplace, and Ivory's Lemma, have numerous generalizations, old and new.  The reader interested in the history of this topic is referred to \cite{Tod}. 

This paper consists of three main parts. 

In Section \ref{bconics}, we relate two subjects: Ivory's Lemma  and billiard dynamics in domains bounded by quadrics. In particular, in Theorem \ref{IvGenPlane}, we deduce planar Ivory's Lemma from complete integrability of billiards bounded by confocal conics. In this approach, Ivory's Lemma follows from a version of the Poncelet Porism (discovered at about the same time, in 1813). 
This approach extends to numerous generalizations of Ivory's Lemma, including its 
multi-dimensional versions in the spherical and hyperbolic geometries. 

One of our main inspirations in Section \ref{bconics} were recent results of A. Akopyan and A. Bobenko \cite{AB15} on the nets of lines whose quadrilaterals admit inscribed circles. In this direction, our billiard approach gives a proof of a theorem of Reye and Chasles (Theorem \ref{inscribed}) and provides a configuration of circles associated with a periodic billiard trajectory in an ellipse (Figure \ref{grid1}). 

Section \ref{LSmetr} concerns more general metrics, Liouville (in dimension 2) and St\"ackel (in higher dimensions), in which an analog of Ivory's Lemma holds. Our main goal here is to bring back to the contemporary reader somewhat lesser-known results of the German school obtained in the late 19th and early 20th centuries (Blaschke, St\"ackel, Weihnacht, Zwirner).

One of these results (Theorem \ref{thm:LiouvIvory}), due to  Blaschke and Zwirner, is  
that the Ivory property (the diagonals of the coordinate quadrilaterals have equal geodesic lengths) is equivalent to the metric having a Liouville form. In Section \ref{sec:StBill}, we return back to billiards and show that the billiard  bounded by coordinate hypersurfaces of a St\"ackel metric is  integrable.

Section \ref{NIcurv} concerns generalizations of Newton's and Ivory's theorems to the spherical and hyperbolic spaces. This subject was relatively recently investigated by V. Kozlov \cite{Koz00}.

We define gravitational (or Coulomb) potential of a point as the function that is harmonic and rotationally invariant. We prove the spherical and hyperbolic version of Newton's theorem (Theorem \ref{thm:NewtonSph}) by a geometric argument, close to Newton's original one.  

Next, we define a homeoid in the $n$-dimensional spherical and hyperbolic space as the shell between two level sets of a quadratic form defined in the ambient $n+1$-dimensional space. Theorem \ref{LIconst} provides a spherical and hyperbolic version of the Laplace theorem on the potential of a homeoid. Our arguments are again geometric and close to the proof of the Laplace theorem given by Ivory and Chasles. 

The theorems of Newton and Laplace were extended by V. Arnold \cite{Arn82,Arn83} to algebraic hypersurfaces in $\R^n$. In Theorem \ref{Arnconst}, we generalize Arnold's result to the spherical and hyperbolic spaces. 

Let us mention another generalization of Newton's and Ivory's theorems, to magnetic fields and quadrics of all signatures, which also goes back to Arnold \cite{Arn84,VS85}. We believe that these results should also have spherical and hyperbolic versions.
\bigskip

{\bf Acknowledgments}. We are grateful to A. Akopyan, A. Bobenko, V. Dragovic, 
D. Khavinson, E. Lundberg, Yu. Suris, A. Veselov for stimulating discussions 
and advice. Part of this work was done at ICERM, Brown University, during the 
second's author 2-year stay there and the first author  visit at the institute. 
We are grateful to ICERM for its inspiring  and encouraging atmosphere. 
The second author was supported by the NSF grant  DMS-1510055.

\section{Billiards, conics, and quadratic surfaces} \label{bconics}

\subsection{Billiards in confocal conics} \label{bconf}

In this section, we recall some basic facts about billiards and conics; see, e.g., \cite{DR11,FT07,KT91,Tab95,Tab05} and, specifically, \cite{LT07}.

We consider billiards as a discrete-time dynamical system acting on oriented lines: an incoming billiard trajectory hits the billiard curve and reflects so that the angle of incidence equals the angle of reflection. Equivalently, one may think in terms of  geometrical optics: oriented lines are rays of light, and the billiard curve is an ideal mirror.

\subsubsection{Invariant area form} \label{invarea}
 
The space of oriented lines has an area form that is preserved by the optical reflections (independently of the shape of the mirror). 

Choose an origin, and introduce coordinates $(\alpha,p)$ on the space of rays: $\alpha$ is the direction of the ray, and $p$ is its signed distance to the origin, see Figure \ref{coord}. Then the invariant area form is as follows: $\omega=d\alpha\wedge d p$.

\begin{figure}[hbtp]
\centering
\includegraphics[height=1.2in]{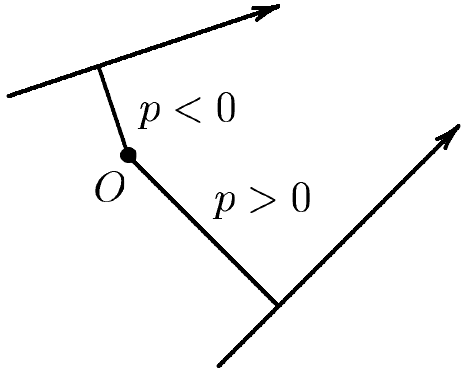}
\caption{Coordinates on the space of rays}
\label{coord}
\end{figure}

This symplectic structure is obtained by symplectic reduction from the canonical symplectic structure of the cotangent bundle $T^* \R^2$. This construction is quite general, and it yields a symplectic structure on the space of oriented non-parameterized geodesics of a Riemannian manifold (assuming that this space is a smooth manifold). For example, this is the case in the spherical and hyperbolic geometries. See, e.g., \cite{Ar89,Tab95,Tab05} for details.

\subsubsection{Caustics and string construction} \label{caust}
A caustic of a billiard is a curve $\gamma$ with the following property: if a segment of a billiard trajectory is tangent to $\gamma$ then so is each reflected segment.  

Consider an oval (closed smooth strictly convex curve) and fix a point $C$ on it. For a point $X$ outside of the oval, consider two functions: 
$$
f(X)=|XA|+\stackrel{\smile}{|AC|},\ g(X)=|XB|+\stackrel{\smile}{|BC|},
$$
see Figure \ref{string}.

\begin{figure}[hbtp]
\centering
\includegraphics[height=2in]{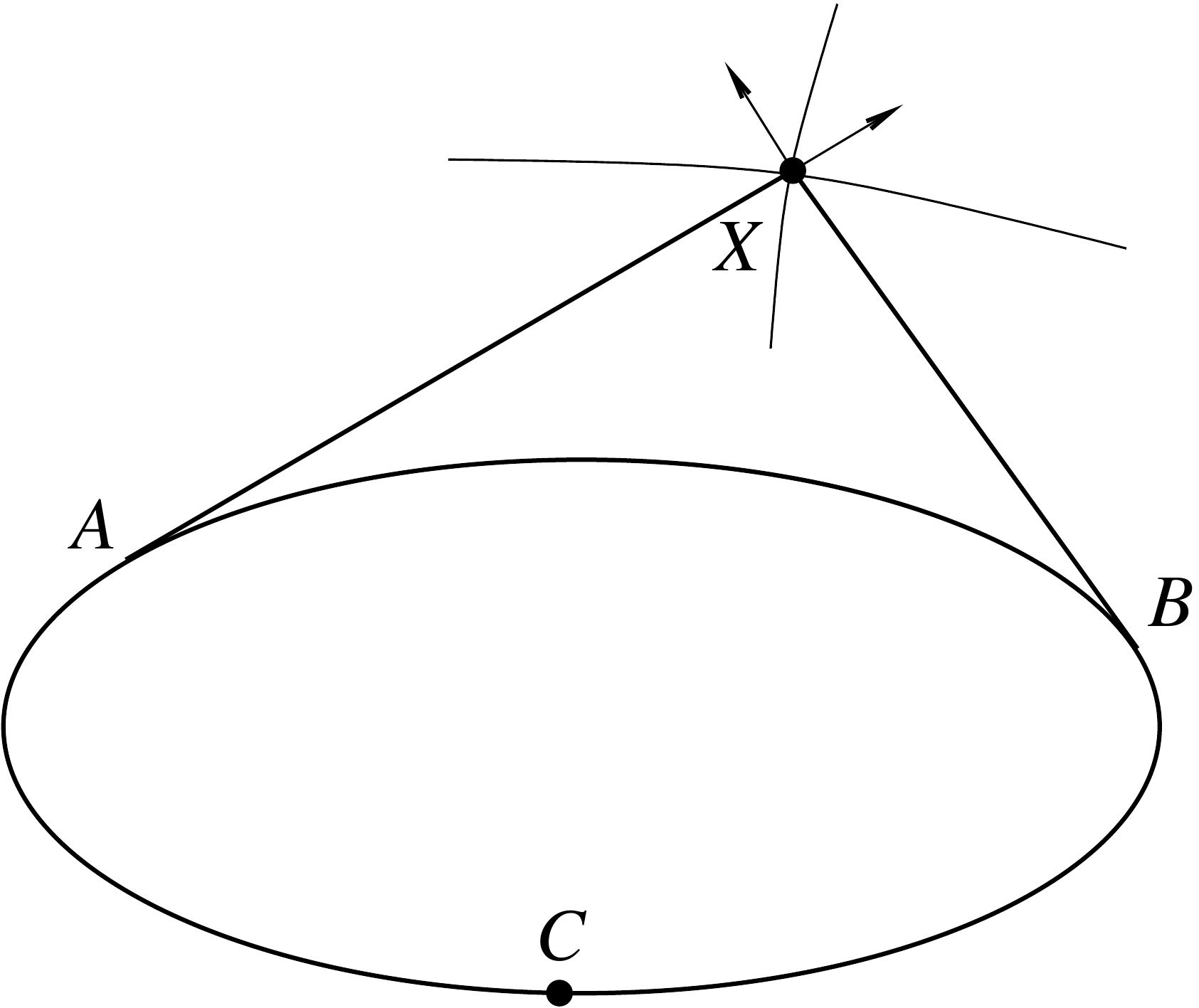}
\caption{String construction}
\label{string}
\end{figure}

The gradients of these functions are the unit vectors along the lines $AX$ and $BX$, respectively. It follows that these two lines make equal angles with the level curves of the functions $f+g$ and $f-g$, and that these level curves are orthogonal to each other.

The function $f+g$ does not depend on the choice of the reference point $C$. Its level curves are given by the string construction: wrap a closed nonelastic string around an oval, pull it tight at a point, and move this point around the oval. This  construction recovers a 1-parameter family of billiard curves from a common caustic (the length of the string is a parameter).

\subsubsection{Billiard properties of conics} \label{conics}

The interior of an ellipse is foliated by confocal ellipses. These are caustics of the billiard inside an ellipse, see, e.g., \cite[Theorem 28.2]{FT07}. Thus one has Graves's theorem: wrapping a closed nonelastic string around an ellipse produces a confocal ellipse. Since confocal ellipses and hyperbolas are orthogonal, an ellipse is also a caustic for reflection in a confocal hyperbola, see Figure \ref{reflection}.

\begin{figure}[hbtp]
\centering
\includegraphics[height=2in]{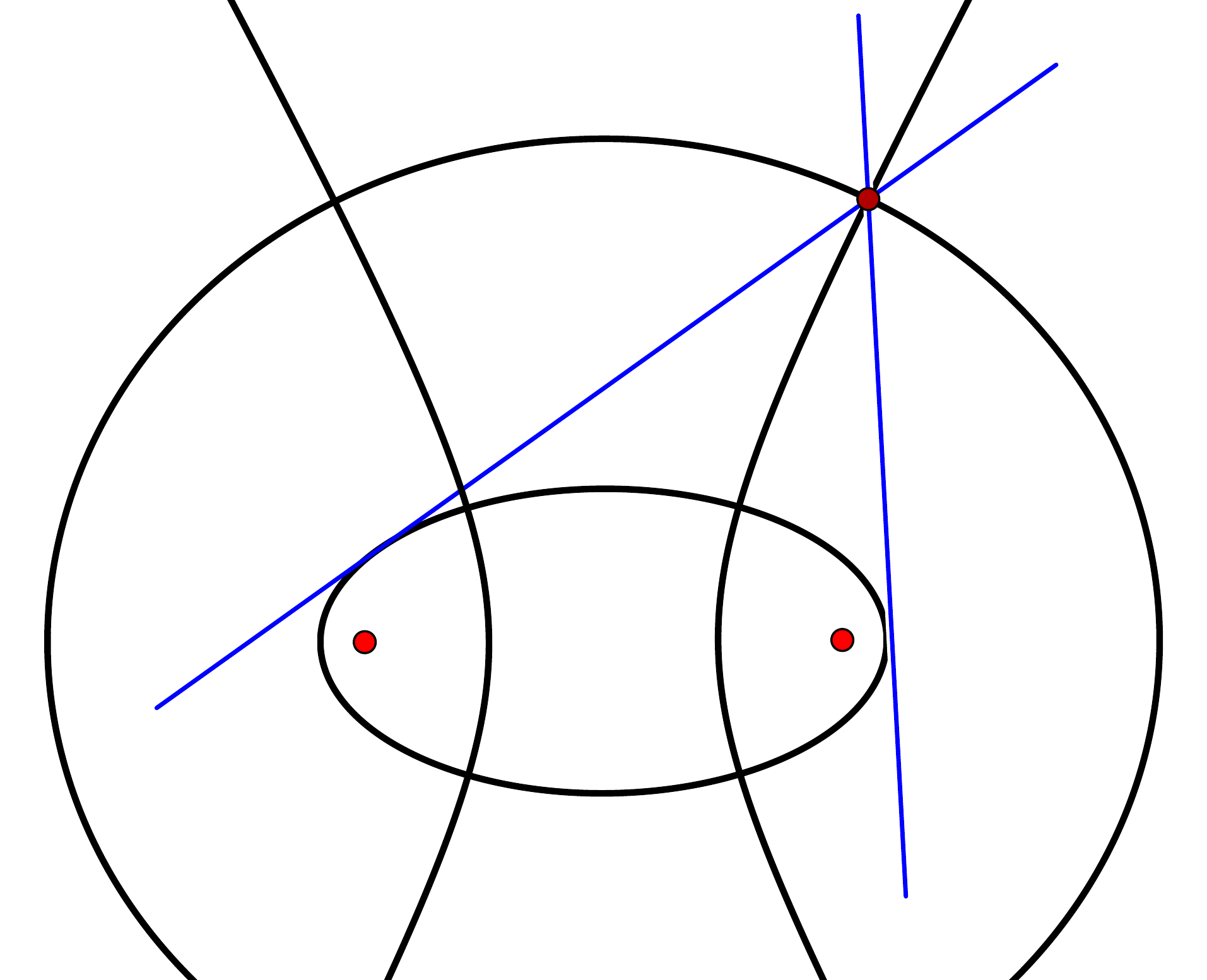}
\caption{Reflection in confocal conics}
\label{reflection}
\end{figure}

If a ray passes between the foci of an ellipse then it is tangent to a confocal hyperbola, and all the reflected rays are tangent to the same confocal hyperbola which, in this case, is a caustic.

\subsubsection{Complete integrability and its consequences} \label{complint}
The space of rays $A$ that intersect an ellipse  is topologically a cylinder, and the billiard system inside the ellipse is an area preserving transformation $T: A\to A$. The cylinder is foliated by the invariant curves of the map $T$ consisting of the rays tangent to confocal conics, see Figure \ref{portrait}. 

\begin{figure}[hbtp]
\centering
\includegraphics[height=1.7in]{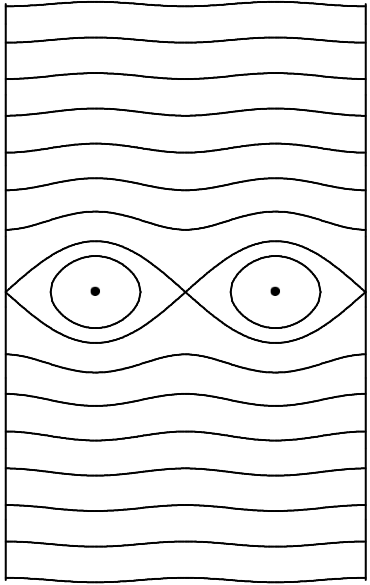}
\caption{Phase portrait of the billiard map in an ellipse}
\label{portrait}
\end{figure}

The curves that go around the cylinder correspond to the rays that are tangent to confocal ellipses, and the curves that form `the eyes' to the rays that are tangent to confocal hyperbolas. A singular curve consists of the rays through the foci, and the two dots to the 2-periodic back and forth orbit along the minor axis of the ellipse.

One can choose a cyclic parameter, say, $x$ modulo 1, on each invariant curve such that the map $T$ becomes a shift $x \mapsto x+c$, where the constant $c$ depends on the invariant curve. This is a manifestation of the Arnold-Liouville theorem in the theory of completely integrable systems, see \cite{Ar89}.

The construction is as follows. Choose a function $H$  whose level curves are the invariant curves that foliate $A$, and consider its Hamiltonian vector field sgrad $H$ with respect to the area form $\omega$. This vector field is tangent to the invariant curves, and the desired coordinate $x$ on these curves is the one in which sgrad $H$ is a constant vector field $d/dx$. Changing $H$ scales the coordinate $x$ on each invariant curve and, normalizing the `length' of the invariant curves to 1, fixes $x$ uniquely up to an additive constant. In other words, the 1-form $dx$ is well defined on each invariant curve. 

The billiard map $T$ preserves the area form and the invariant curves, therefore its restriction to each curve preserves the measure $dx$, hence, is a shift $x \mapsto x+c$.

An immediate consequence is the Poncelet Porism: if a billiard trajectory in an ellipse closes up after a number of reflections then all trajectories with the same caustic close up after the same number of reflections (the general form of the Poncelet Porism is obtained by applying a projective transformation to a pair of confocal conics). 

Note that the invariant measure $dx$ on the invariant curves does not depend on the choice of the billiard curve from 
a confocal family: all confocal ellipses share their caustics. This implies that the billiard transformations with respect to two confocal ellipses commute: restricted to a common caustic, both are shifts in the same coordinate system. See Figure \ref{commute}.

\begin{figure}[hbtp]
\centering
\includegraphics[height=2in]{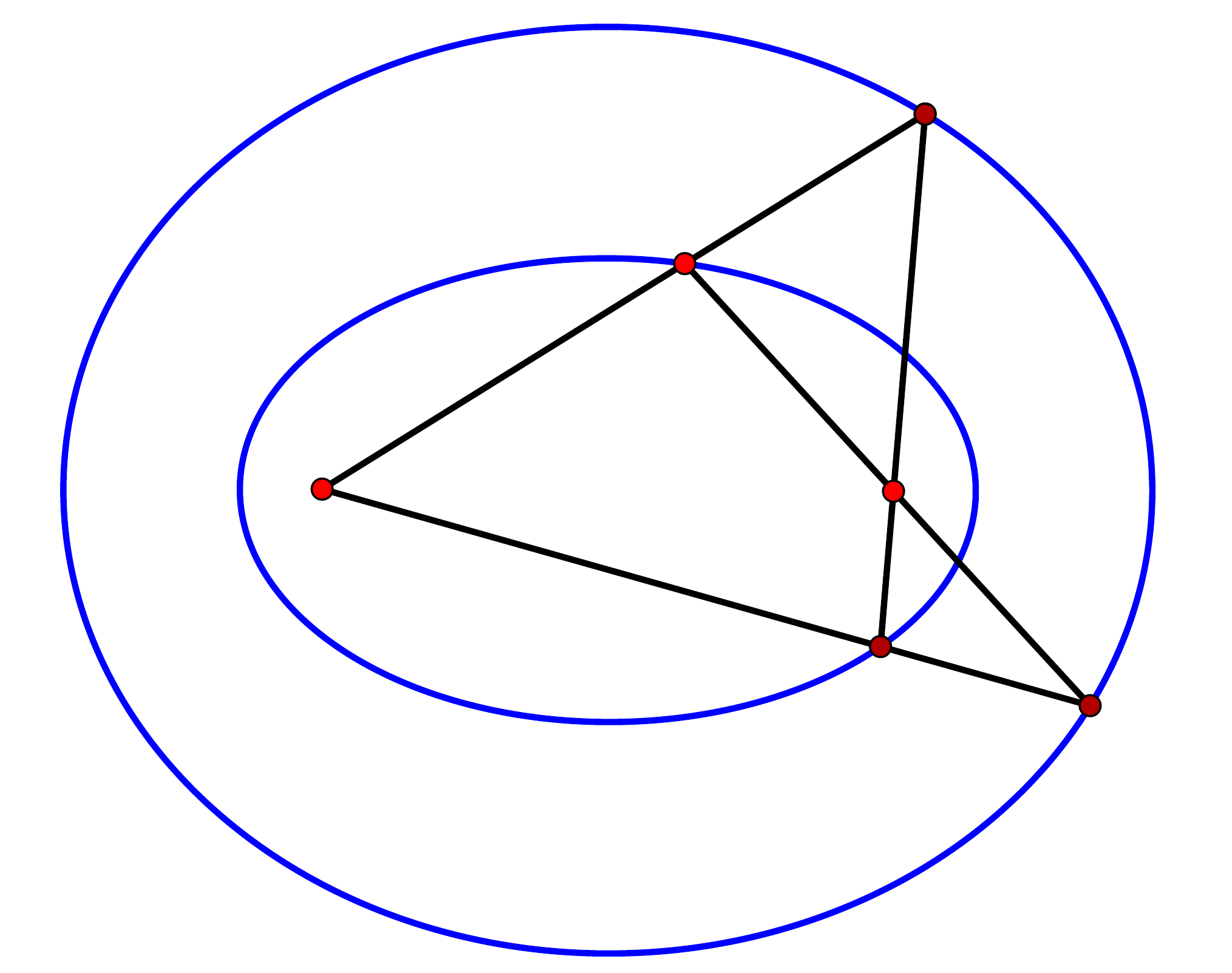}
\includegraphics[height=2in]{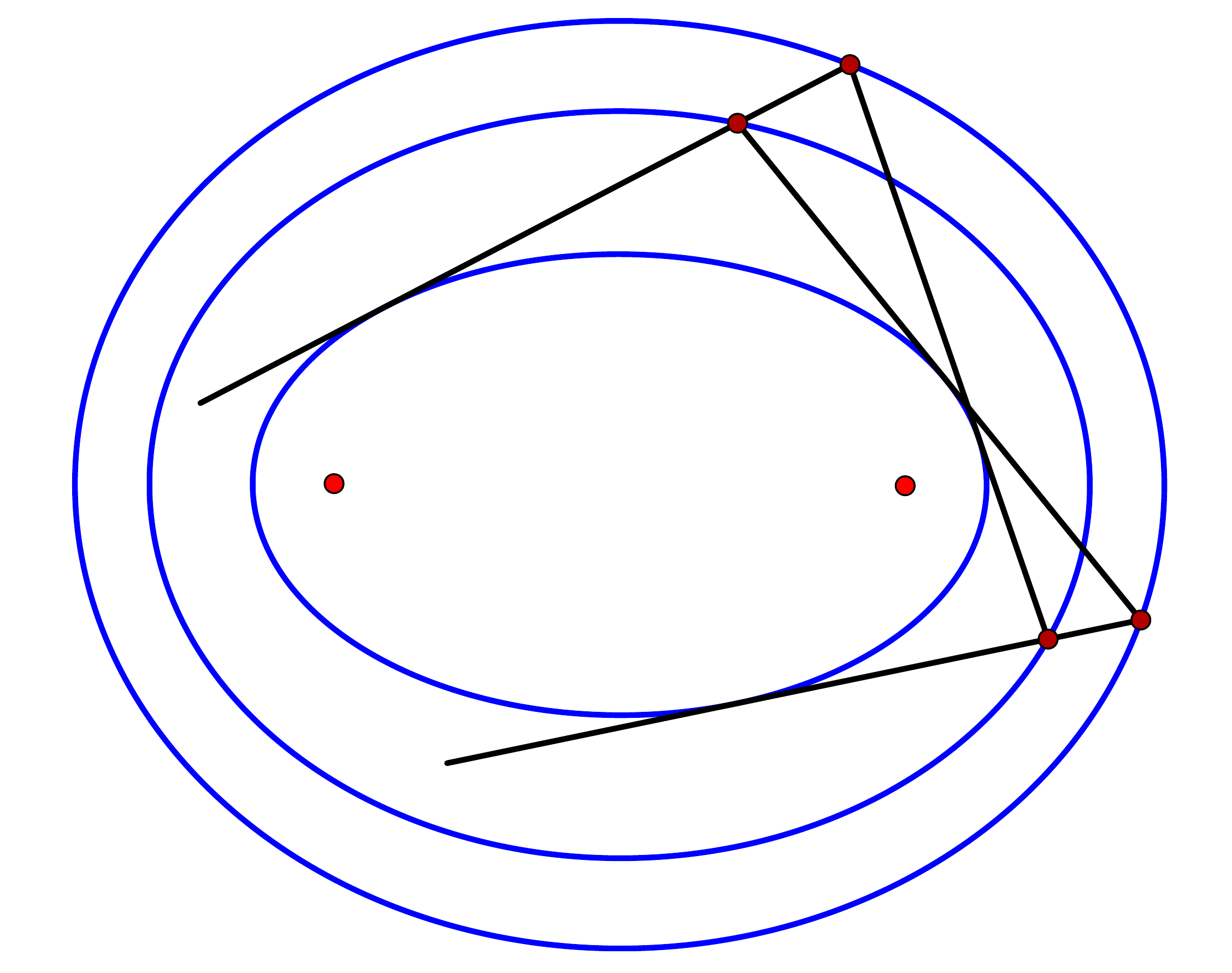}
\caption{Left: the billiard reflections of the rays from a focus in two confocal ellipses commute. Right: the general case.}
\label{commute}
\end{figure}

We identify the invariant curves with the respective conics and  refer to the coordinate $x$ as the canonical coordinate.

\subsubsection{Coordinates in the exterior of a conic} \label{coordext}
Consider an ellipse $\gamma$, and let $x$ be the canonical  coordinate on it. This makes it possible to define coordinates in the exterior of the ellipse: the coordinates of a point $X$ outside of $\gamma$ are the coordinates $x_1$ and $x_2$ of the tangency points of the tangent lines from $X$ to $\gamma$ (points $A$ and $B$ in Figure \ref{string}).

The above discussion implies that the confocal ellipses are given by the equations $x_2-x_1=$ const. Not surprisingly, the confocal hyperbolas have the equations $x_2+x_1=$ const. We repeat an argument from \cite{LT07} here.

\begin{figure}[hbtp]
\centering
\includegraphics[height=2.3in]{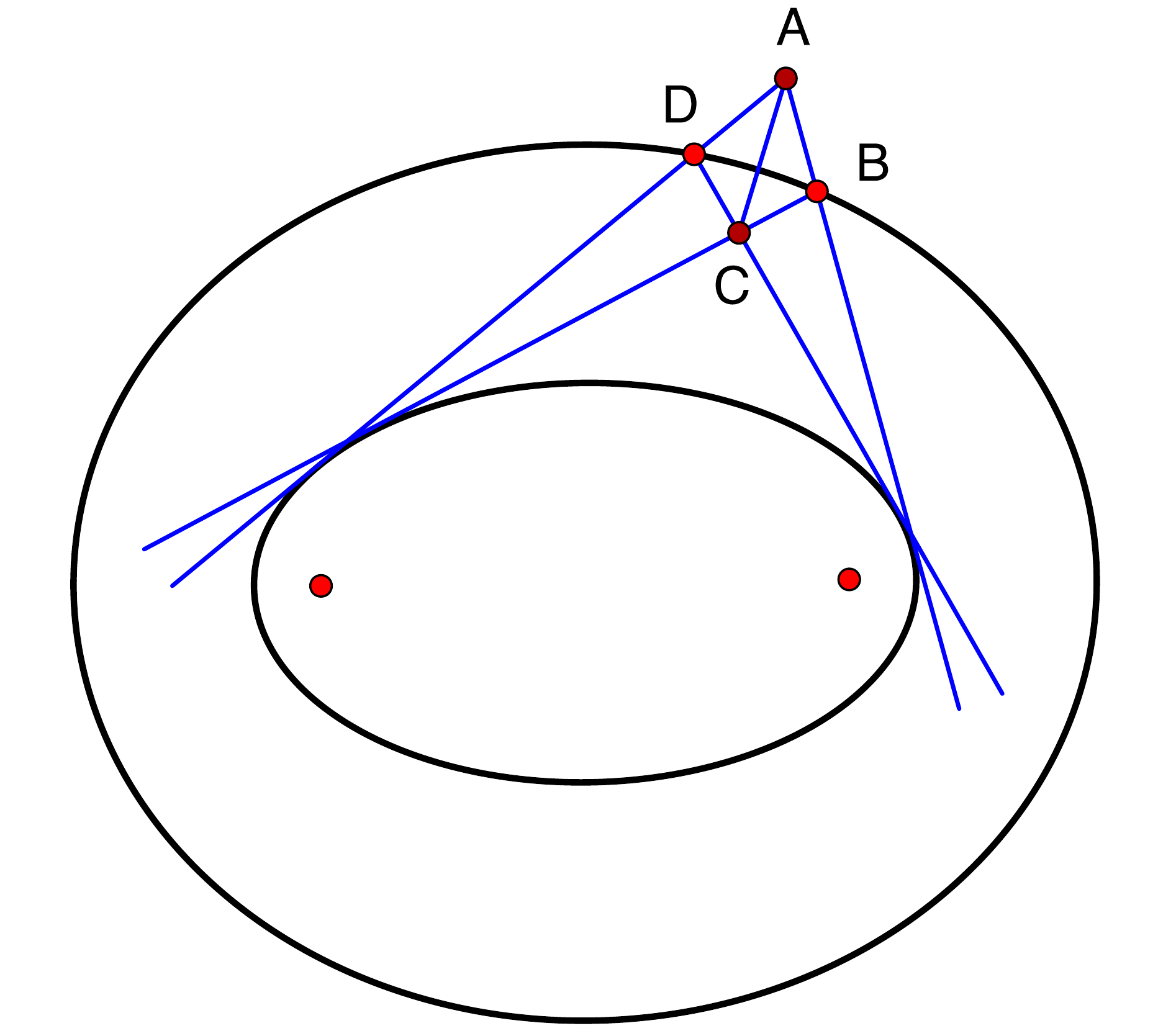}
\caption{Proving that points $A$ and $C$ lie on a confocal hyperbola.}
\label{coordinates}
\end{figure}

Let the coordinates of the tangency points on the inner ellipse, from left to right, be $x_1,x_2,x_3,x_4$, so that 
$$
A(x_1,x_4),\ B(x_2,x_4),\ C(x_2,x_3),\ D(x_1,x_3),
$$
see Figure \ref{coordinates}. Since $B$ and $D$ are on a confocal ellipse, $x_4-x_2=x_3-x_1$, and hence $x_2+x_3=x_1+x_4$.

By the billiard property, the arc of an ellipse $BD$ bisects the angles $ABC$ and $ADC$. Therefore, in the limit  $D\to B$, the infinitesimal quadrilateral $ABCD$ becomes a kite: the diagonal $BD$ is its axis of symmetry. Hence $AC \perp BD$, and the locus of points given by the equation $x_1+x_4=$ const and containing points $A$ and $C$  is orthogonal to the ellipse through points $B$ and $D$. Therefore this locus is a confocal hyperbola.

Let us summarize. Consider Figure \ref{reflection} again. In the coordinate $x$ on the inner ellipse, the billiard  reflections in a confocal ellipse and a confocal hyperbola are given, respectively, by the formulas
\begin{equation}
\label{tworefl}
x \mapsto x + a,\quad x\mapsto b-x,
\end{equation}
where the constants $a$ and $b$ depend on the choice of the outer ellipse and the hyperbola.

\subsection{Ivory's Lemma in the plane} \label{plane}
\subsubsection{Proof by billiards} \label{pfbill}
Recall the statement of Ivory's Lemma: the diagonals of a quadrilateral made by arcs of confocal ellipses and hyperbolas are equal: in Figure \ref{IvL}, $|AC|=|BD|$.

The idea of the proof is very simple. Imagine that we need to prove that the diagonals of a rectangle are equal.
Let us consider a diagonal as  a 4-periodic billiard trajectory.

\begin{figure}[hbtp]
\centering
\includegraphics[height=0.8in]{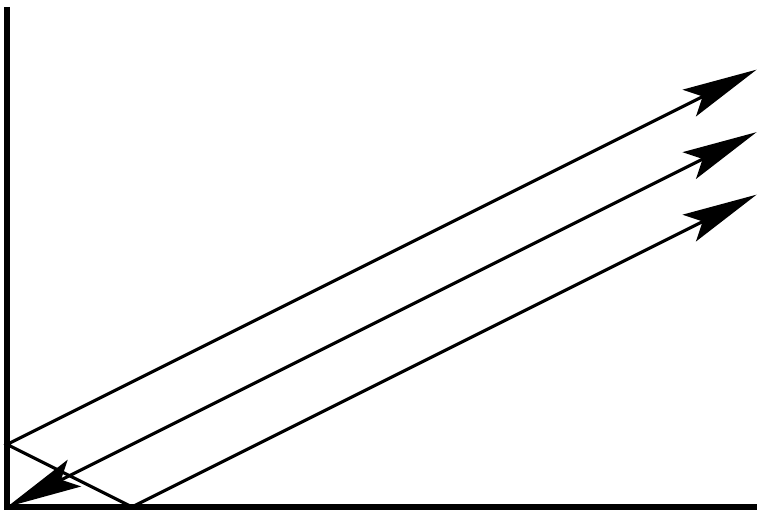}
\caption{Reflection in a right angle}
\label{corner}
\end{figure}

In general, a billiard trajectory that hits a corner cannot be continuously extended, but if the angle is $90^{\circ}$, such an extension is possible. Indeed, an infinitesimally close parallel trajectory, entering an angle, makes two reflections and exits in the opposite direction, see Figure \ref{corner}.  This holds both for the trajectories on the right and on the left of the `dangerous' trajectory that goes directly to the corner. This makes it possible to define reflection in a right angle as the direction reversal of the ray.

To conclude that the diagonals are equal, include a diagonal into a 1-parameter family of 4-periodic trajectories, interpolating between the two diagonals, and observe that these trajectories have
the same perimeter lengths, see Figure \ref{rectangle}. 

\begin{figure}[hbtp]
\centering
\includegraphics[height=0.8in]{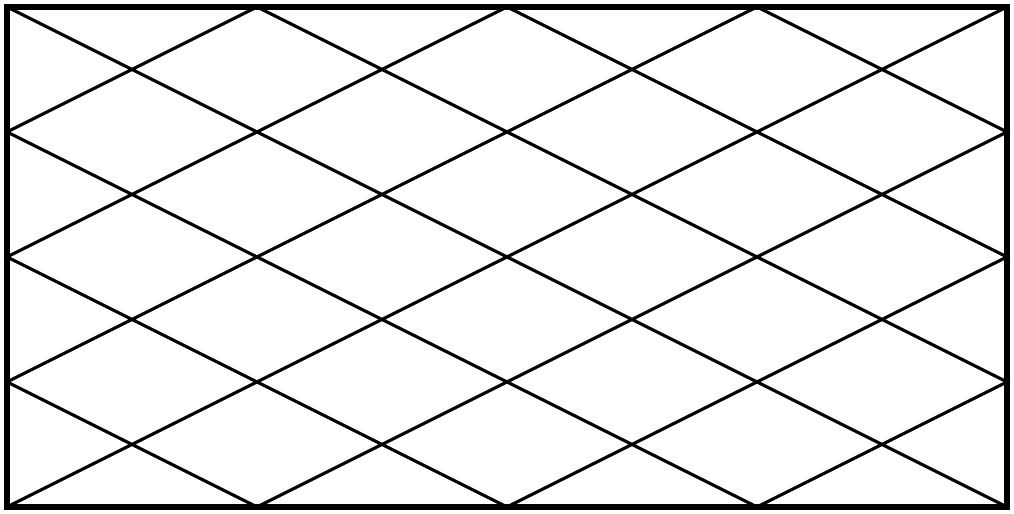}
\caption{The diagonals of a rectangle include in a family of 4-periodic billiard trajectories}
\label{rectangle}
\end{figure}

We use a similar argument with  the curvilinear quadrilateral $ABCD$ made of confocal conics. The following porism implies Ivory's Lemma.

\begin{theorem} \label{IvGenPlane}
Let $\gamma$ be the  conic from the confocal family that is tangent to the line $BD$ (the inner ellipse in Figure \ref{IvL}). There exists a 1-parameter family of 4-periodic billiard trajectories in the quadrilateral $ABCD$, interpolating between the diagonals $BD$ and $AC$, and consisting of rays tangent to $\gamma$ (such as the quadrilateral $PQRS$ in Figure \ref{IvL}). In particular, the line $AC$ is also tangent to $\gamma$.
These 4-periodic trajectories have the same perimeter lengths, and hence, $|AC|=|BD|$.
\end{theorem}

\begin{figure}[hbtp]
\centering
\includegraphics[height=2.3in]{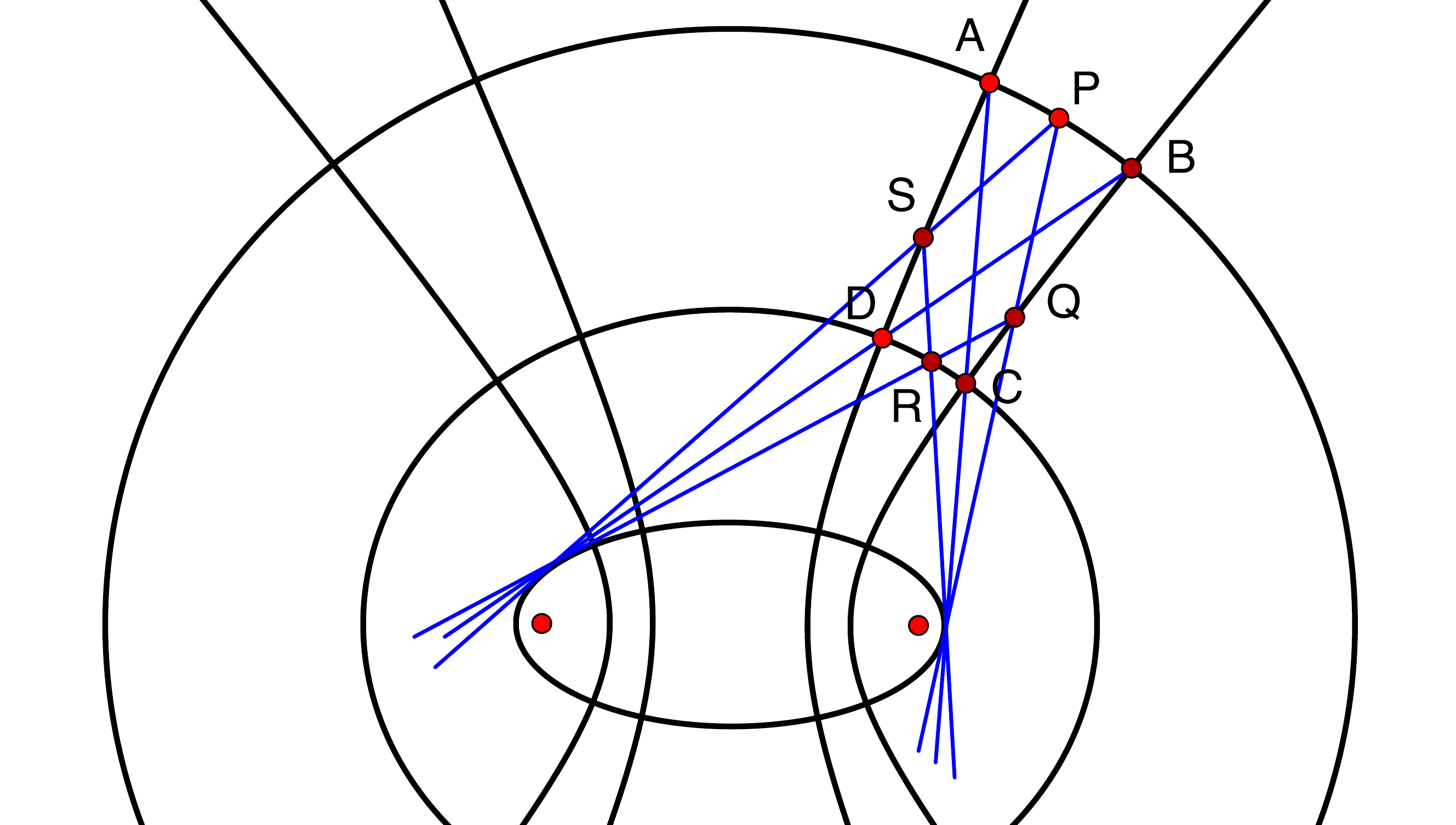}
\caption{Proof of Ivory's Lemma}
\label{IvL}
\end{figure}

\proof We consider the case when the line $BD$ is tangent to a confocal ellipse, as in Figure \ref{IvL}. The case of a confocal hyperbola is similar, and the intermediate case when the line passes through a focus is obtained as a limit.

Consider a ray tangent to $\gamma$ and its four consecutive reflections in the sides of the quadrilateral $ABCD$. According to the discussion in Section \ref{bconics}, each reflected ray is again tangent to $\gamma$.  

Let $x$ be the canonical coordinate on $\gamma$. 
According to (\ref{tworefl}), these reflections are given by formulas 
$$
x \mapsto x + a,\quad x\mapsto b-x, \quad x \mapsto x + c,\quad x\mapsto d-x,
$$
where the constants depend on the conics that form the quadrilateral. 

The composition of these maps is a shift $x\mapsto x+(a-b-c+d)$, and if it has a fixed point then it is the identity. But the 4-periodic trajectory $BD$ provides a fixed point, whence a 1-parameter family of 4-periodic trajectories.

An $n$-periodic billiard trajectory is a critical point of the perimeter length function $L$ on the space of inscribed $n$-gons. A 1-parameter family of such trajectories is a curve consisting of critical points of $L$. It follows that $dL$ vanishes on this curve, and hence the value of $L$ remains constant.
\proofend

\subsubsection{Inscribed circles} \label{inscrcirc}
Using results from Section \ref{bconics}, we obtain the following theorem that goes back to Reye and Chasles, see also \cite{AB15}.

\begin{theorem} \label{inscribed}
Let $A$ and $B$ be two points on an ellipse. Consider  the quadrilateral $ABCD$, made by the pairs of tangent lines from $A$ and $B$ to a confocal ellipse. Then its other vertices, $C$ and $D$, lie on a confocal hyperbola, and the quadrilateral 
 is circumscribed about a circle,  see Figure \ref{circlequad}.
\end{theorem}

\begin{figure}[hbtp]
\centering
\includegraphics[height=2.5in]{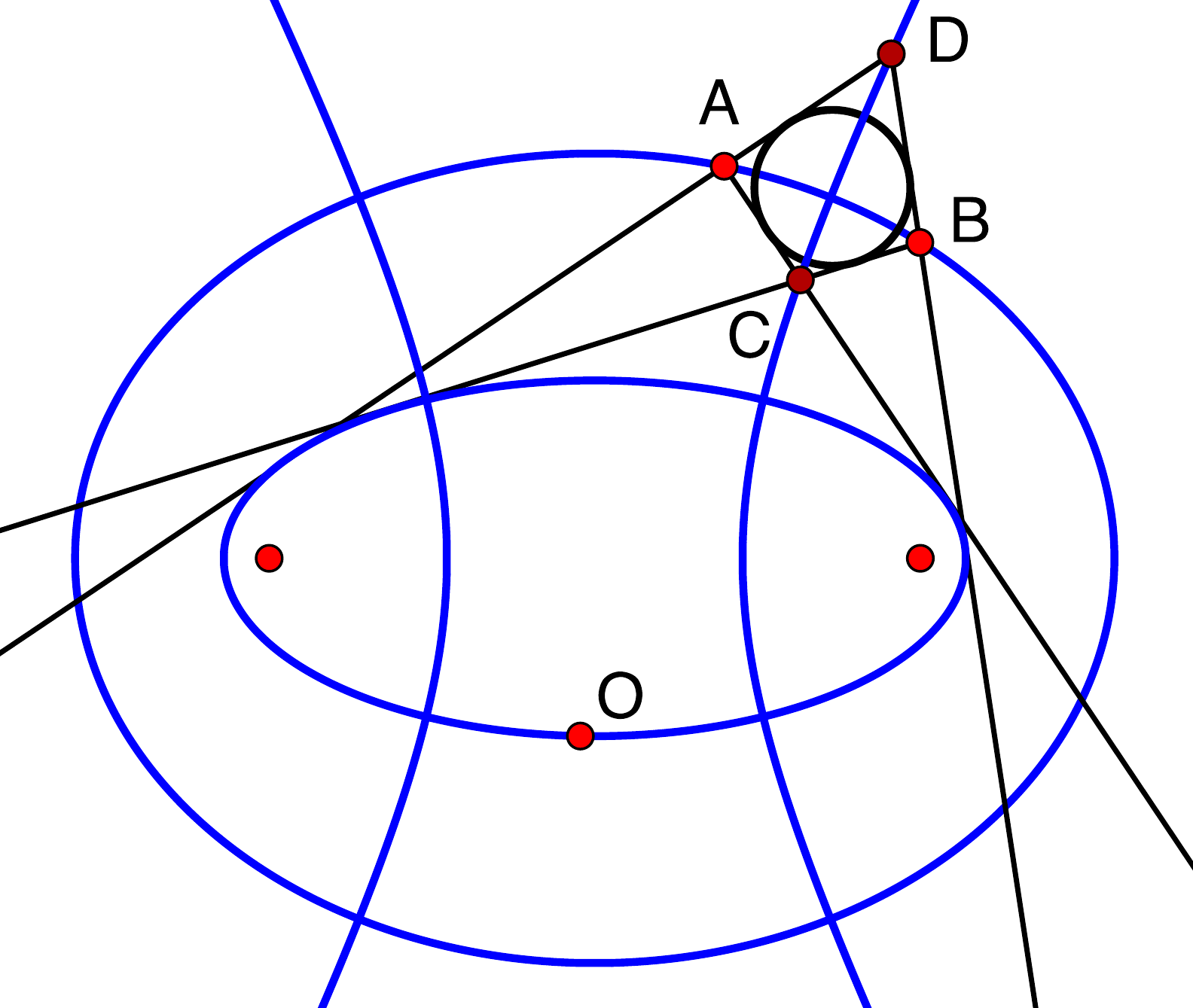}
\caption{Circumscribed quadrilateral}
\label{circlequad}
\end{figure}

\proof Let $x_1, x_2, y_1, y_2$ be the canonical coordinates of the tangency points of the lines $AD, BC, AC, BD$ with the inner ellipse. Since points $A$ and $B$ lie on a confocal ellipse, $y_1-x_1=y_2-x_2$, see Section \ref{coordext}. Then $y_1+x_2=y_2+x_1$, and hence points $C$ and $D$ lie on a confocal hyperbola.

Choose a point $O$ on the inner ellipse and consider the respective functions $f$ and $g$ introduced in Section \ref{caust}. Then 
$$
f(A)+g(A)=f(B)+g(B),\ \ f(C)-g(C)=f(D)-g(D),
$$
hence
$$
f(D)-f(A) - g(A)+g(C)+f(B)-f(C)-g(D)+g(B)=0,
$$
or
\begin{equation} \label{perimeter}
|AD|-|AC|+|BC|-|BD|=0.
\end{equation}
This is necessary and sufficient for the quadrilateral $ABCD$ to be circumscribed.
\proofend 

\subsubsection{Poncelet grid of circles} \label{grid}
Poncelet grid consists of the intersection points of the sides of a Poncelet polygon, that is, a polygon which is inscribed into an ellipse and circumscribed about an ellipse. The points of this grid can be arranged into `concentric' subsets that lie on ellipses and into `radial' subsets that lie on hyperbolas
%; all the concentric subsets are projectively equivalent, and so are all the radial ones
. See \cite{Sch07}. 

A pair of nested ellipses is projectively equivalent to a pair of confocal ones; this was used to prove the properties of the Poncelet grid in \cite{LT07}. In this confocal case, each concentric set lies on a confocal ellipse, and hence each quadrilateral of the grid is circumscribed,  see Figure \ref{grid1}. 

\begin{figure}[hbtp]
\centering
\includegraphics[height=3in]{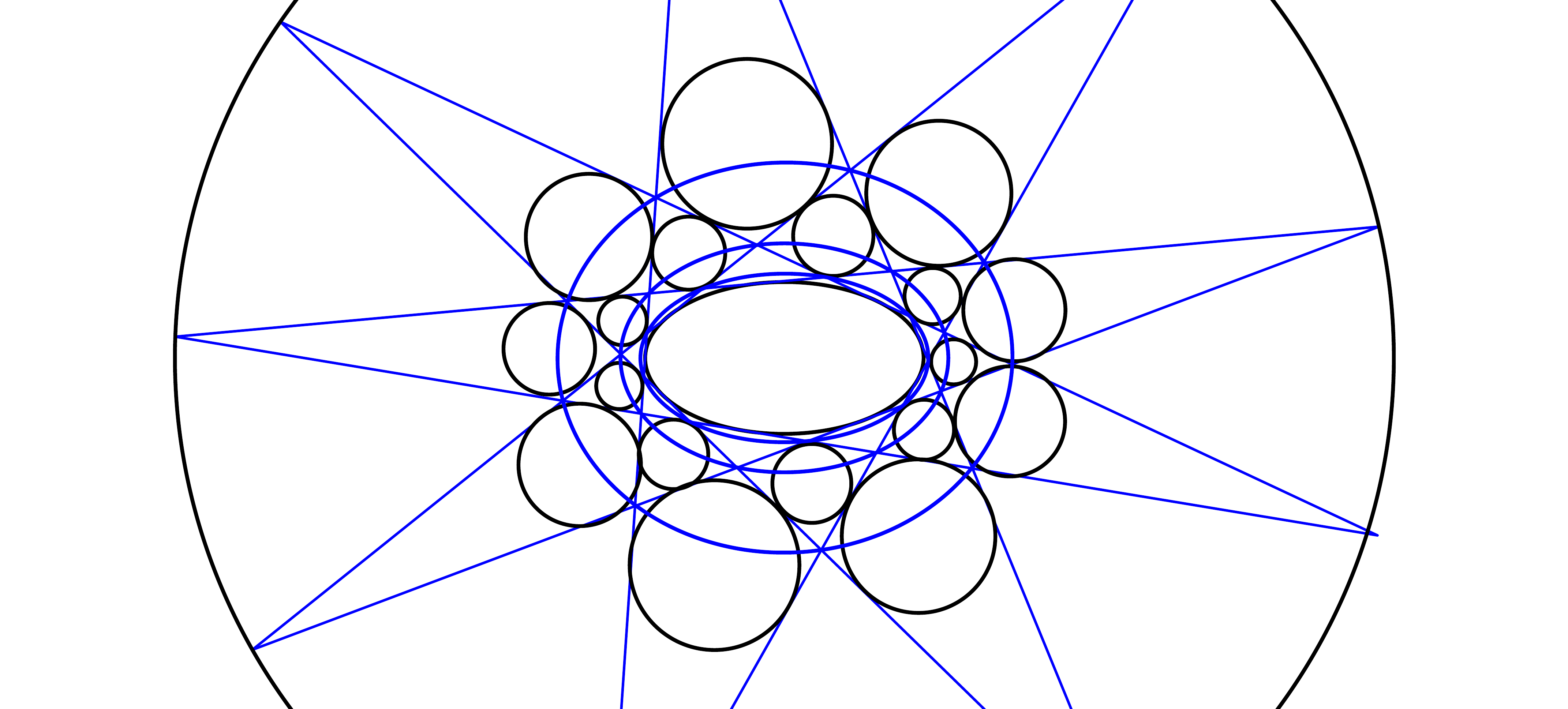}
\caption{Poncelet grid, $n=9$}
\label{grid1}
\end{figure}

\subsection{Ivory's Lemma on quadratic surfaces} \label{surf}

\subsubsection{On ellipsoids} \label{ellipsoids}
Consider a 3-axial ellipsoid
\begin{equation} \label{ellpsd}
\frac{x^2}{a} + \frac{y^2}{b}+ \frac{z^2}{c} =1,\quad a>b>c>0,
\end{equation}
included into the confocal family of quadrics
\begin{equation} \label{conf}
\frac{x^2}{a-\lambda} + \frac{y^2}{b-\lambda}+ \frac{z^2}{c-\lambda} =1.
\end{equation}
Let us recall some classical facts about geometry of quadrics; see, e.g., \cite{Ber87,HC52} and the previously cited books.  

The curves of intersection of the ellipsoid with the confocal quadrics are its lines of curvature, see Figure \ref{ellipsoid}  (borrowed from \cite{HC52}). The singular points of this orthogonal system of curves are the four umbilical points. These points play the role of foci, and the lines of curvature the role of confocal ellipses and hyperbolas. In particular, the lines of curvature are the loci of points whose sum of geodesic distances to a pair of non-antipodal umbilical points is constant.

\begin{figure}[hbtp]
\centering
\includegraphics[height=2in]{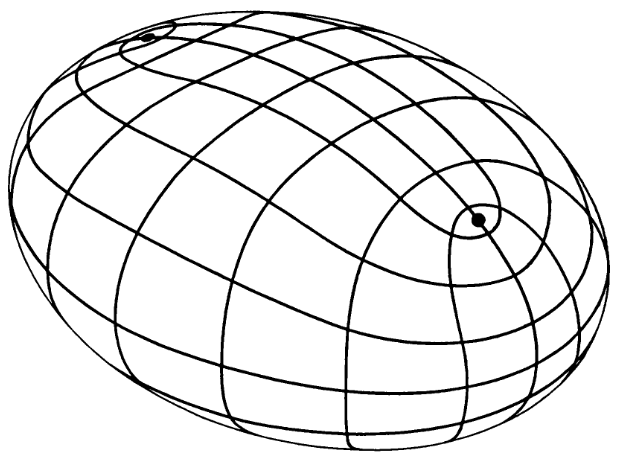}
\caption{Lines of curvature on an ellipsoid}
\label{ellipsoid}
\end{figure}

Consider the billiard inside a domain bounded by a line of curvature (the rays being geodesic and the reflection  optical). This system is completely integrable, as we describe below. See \cite{AF06,ChSh89,GKhT07,Ves90,Ves91} concerning billiards on quadratic surfaces.

According to a Chasles theorem, a generic line is tangent to two quadrics from the confocal family (\ref{conf}). 
The geodesic flow on the ellipsoid has the following property: the tangent lines to a geodesic curve remain tangent to a fixed confocal quadric. The billiard reflection from a quadric in $\R^3$ has the property that the incoming and the outgoing rays are tangent to the same pair of confocal quadrics. 

Combined, these facts  imply that the consecutive geodesic segments of a billiard trajectory remain tangent to the same line of curvature, which serves as a billiard caustic. Thus the situation is similar to the planar one, as described  in Section \ref{bconics}. In particular, one has an ellipsoid version of the Graves theorem: wrapping a closed nonelastic string around a line of curvature produces a line of curvature.

The area-preserving (symplectic) property of the billiard map holds as well: the space of geodesic chords of the billiard table has a canonical area form, invariant under the billiard reflection. The arguments from Section \ref{plane} apply with minimal adjustments, yielding the next result.

\begin{theorem}
\label{IvCurved}
For a quadrilateral made of lines of curvature of a triaxial ellipsoid, the two pairs of opposite vertices are at equal geodesic distances.
\end{theorem}

\begin{remark}[Doubly ruled surfaces] \label{hyperboloid}
{\rm 
The hyperboloid of one sheet is a doubly ruled surface, and the rulings are geodesics. One has a variant of Theorem \ref{IvCurved}, see \cite{AB15}: 
{\it Consider a curvilinear quadrilateral $ABCD$ on a hyperboloid of one sheet whose sides are curvature lines, and points $A$ and $C$ lie on a ruling. Then points $B$ and $D$  lie on a ruling from a different family, and $|AC|=|BD|$.}

This can be proved using the same billiard approach: the rulings, being asymptotic lines, make equal angles with the curvature lines, and hence form segments of billiard trajectory in a table bounded by curvature lines.

Likewise for another doubly ruled surface, the hyperbolic paraboloid.
}
\end{remark}

%{\bf Next remark needs to be removed/replaced}

%\begin{remark}[Liouville surfaces] \label{Liouv}
%{\rm Theorem \ref{IvCurved} extends to Liouville surfaces whose metric is given, in local coordinates, by
%\begin{equation} \label{Lio}
%ds^2=(U(u)-V(v))\ (U_1(u) du^2 + V_1(v) dv^2).
%\end{equation}
%Central quadrics are examples of Liouville surfaces.

%The geodesic diagonals of the quadrilaterals made of the coordinate lines are equal, and conversely, if a surface has a parametrization that has this Ivory property with respect to the coordinate lines, then the metric in this parametrization has the Liouville form (\ref{Lio}), see \cite{Zw,Bla28}.

%Our billiard approach works in this case as well: the billiard on a Liouville surface bounded by coordinate lines is integrable. This result goes back to Darboux \cite[Livre VI, Chapitre I]{Da1}; we refer to \cite{DR11} section 5.5, for a modern account.
%}
%\end{remark}

\subsubsection{On the sphere and in the hyperbolic plane} \label{spherehyp}
The notion of confocal spherical conics is  classical, see \cite{Da}: these are the intersections of the unit sphere with the confocal family of quadratic cones 
\begin{equation} \label{cone}
\frac{x^2}{a-\lambda} + \frac{y^2}{b-\lambda}+ \frac{z^2}{c-\lambda} =0,\quad a>b>c,\ a>\lambda>c.
\end{equation}
Here $a,b$ and $c$ are fixed, and $\lambda$ is a  parameter in the family.

Formula (\ref{cone}) is obtained from (\ref{conf}) in the limit $a,b,c \to 1$.
For the confocal quadrics (\ref{conf}) to intersect the ellipsoid (\ref{ellpsd}), one must have $a>\lambda>c$, and hence $\lambda \to 1$ as well. In the limit, the ellipsoid (\ref{ellpsd}) becomes the unit sphere  $x^2+y^2+z^2=1$, and its intersections with the confocal quadrics become the curves given by (\ref{cone}). 

The spherical billiard inside a domain, bounded by a spherical conic, is integrable in the same way as in the plane, see \cite{ChSh89,Ves90}. 

A similar approach works in the case of the hyperbolic plane, realized as the pseudosphere $x^2+y^2-z^2=-1$ in the pseudo-Euclidean space $\R^{2,1}$. 

Replace the ellipsoid (\ref{ellpsd}) in $\R^3$ by the hyperboloid of two sheets in $\R^{2,1}$
\begin{equation} \label{hyperb}
-\frac{x^2}{a} - \frac{y^2}{b}+ \frac{z^2}{c} =1,
\end{equation}
with distinct positive $a,b,c$. The respective (pseudo)confocal family is
\begin{equation} \label{pseconf}
-\frac{x^2}{a-\lambda} - \frac{y^2}{b-\lambda}+ \frac{z^2}{c-\lambda} =1,
\end{equation}
whose intersections with the hyperboloid (\ref{hyperb}) are its lines of curvature. 
The billiard inside a domain, bounded by a line of curvature, is integrable, with lines of curvature serving as caustics.

Next, one considers the limit $a,b,c\to 1$ that yields the pseudosphere. Again, $\lambda\to 1$ as well, and one  defines  hyperbolic confocal conics as the intersections of the pseudosphere with the family of quadratic cones 
$$
\frac{x^2}{a-\lambda} + \frac{y^2}{b-\lambda}- \frac{z^2}{c-\lambda} =0,
$$  
where $a,b,c$ are fixed and $\lambda$ is a parameter. See \cite{ChSh89,Ves90}, and \cite{DR12,KhT09} for general information about confocal quadrics in pseudo-Euclidean spaces.

The billiard inside a domain bounded by a hyperbolic conic is integrable as well. As before, this yields 
spherical and hyperbolic versions of Ivory's Lemma:

\begin{theorem}
\label{IvSphHyp}
The diagonals of a quadrilateral made of confocal spherical or hyperbolic conics are equal.
\end{theorem}

An analog of Theorem \ref{inscribed} holds true in the spherical or hyperbolic cases as well. Equation (\ref{perimeter}) is deduced the same way as before, and it is still necessary and sufficient for a quadrilateral to be circumscribed.

See \cite{SW04,Hor11} for Ivory's Lemma in the hyperbolic geometry and, more generally, in Lorentzian space-forms.

\subsection{In higher dimensions} \label{multidim}

In this section, we discuss multi-dimensional versions of Ivory's Lemma.

An ellipsoid with distinct axes in $\R^n$
\begin{equation} \label{ellmult}
\frac{x_1^2}{a_1} +\frac{x_2^2}{a_2}+ \ldots + \frac{x_n^2}{a_n} =1,\quad a_1>a_2>\ldots >a_n>0, 
\end{equation}
is included into the family of confocal quadrics
\begin{equation} \label{conffam}
\frac{x_1^2}{a_1-\lambda} +\frac{x_2^2}{a_2-\lambda}+ \ldots + \frac{x_n^2}{a_n-\lambda} =1.
\end{equation}

The theory of confocal quadrics comprises the following results, see, e.g., \cite{Ar89,Ber87,DR11,Mos80,Tab95,Tab05}:
\begin{itemize}
\item The space of oriented lines in $\R^n$ is a symplectic $2n-2$-dimensional manifold, symplectomorphic to $T^* S^{n-1}$. The billiard reflection in a smooth hypersurface is a symplectic transformation.
\item Through a generic point in space there pass $n$ pairwise orthogonal confocal quadrics (Jacobi). These $n$ quadrics have different topology (for $n=3$, ellipsoid, hyperboloid of one sheet, and hyperboloid of two sheets).
\item A generic line is tangent to $n -1$ confocal quadrics whose tangent hyperplanes at the points of tangency with the line are pairwise orthogonal (Chasles).
\item The set of oriented lines tangent to $n-1$ fixed confocal quadrics is a Lagrangian submanifold of the space of lines.
\item Consider the billiard reflection in one of the confocal quadrics (\ref{conffam}). Then the incoming and outgoing rays  are tangent to the same $n-1$  quadrics from the confocal family.
\item The tangent lines to a geodesic on an ellipsoid are tangent to fixed $n -2$ confocal quadrics (Jacobi-Chasles). 
\end{itemize}

As a result, the billiard map inside a quadric and the geodesic flow on a quadric are completely integrable systems that share the integrals. 

The  space of rays is foliated by invariant Lagrangian submanifolds, and the leaves carry a canonical flat structure (as asserted by the Arnold-Liouville theorem). The billiard map and the geodesic flow preserve this flat structure; in particular, the geodesic flow is a constant vector field in appropriate coordinates on each invariant manifold.  This is a multi-dimensional version of the planar results, described in detail in Section \ref{bconics}.

Consider a `parallelepiped' $\Pi$ bounded by confocal quadrics. 
\begin{theorem} \label{Ivmulti}
The great diagonals of $\Pi$ are equal.
\end{theorem}

\proof The argument is similar to the one in Section \ref{plane}; we present the main steps. 

Since the facets of $\Pi$ are orthogonal, one can  define the billiard reflection at non-smooth points of the boundary of $\Pi$. In particular, a ray that hits a vertex exits in the opposite direction after $n$ (infinitesimal) reflections.\footnote{This is how corner reflectors work; see \url{http://en.wikipedia.org/wiki/Corner_reflector}.}  Thus a diagonal of $\Pi$ is a $2n$-periodic billiard trajectory.

The diagonal is tangent to $n-1$ confocal quadrics. Consider the Lagrangian submanifold $L$ consisting of the lines that are tangent to these $n-1$ quadrics. $L$ is invariant under the billiard reflections in the facets of $\Pi$, and the composition of these $2n$ reflections is a parallel translation on $L$. Since the composition has a fixed point, the diagonal, it is the identity. As a result, one has an $n-1$-parameter family of $2n$-periodic billiard trajectories that includes the diagonal.

The combinatorics of this family is the same as if $\Pi$ was a Euclidean cube (compare with Figure \ref{rectangle}). In particular, this family includes all great diagonals of $\Pi$. Since the length of periodic billiard trajectories  in a family is constant, the great diagonals of $\Pi$ have the same lengths.   
\proofend

\begin{remark}
\label{rem:LinearMap}
{\rm The Ivory Lemma, as formulated in Section \ref{intro}, follows from the 3-dimensional case of Theorem \ref{Ivmulti} and the next additional statement. Let $E_1$ and $E_2$ be two confocal ellipsoids, and let $A$ be a linear map that takes $E_1$ to $E_2$. Let $P_1\in E_1$ and $P_2=A(P_1)\in E_2$ be two corresponding points. Then $P_1$ and $P_2$ lie on the same $n-1$ confocal quadrics, and likewise for a pair of points $Q_1$ and $Q_2$.
}
\end{remark}

\begin{remark}
\label{rem:ExtrIntrDiag}
Consider Figure \ref{ellipsoid} again. By Theorem \ref{IvCurved}, the lengths of geodesic diagonals in each coordinate quadrilateral are equal. On the other hand, the lengths of extrinsic diagonals (segments in $\R^3$) are also equal. This is a limit case of Theorem \ref{Ivmulti} when the two ellipsoids bounding the parallelepiped $\Pi$ coincide.
\end{remark}

Theorem \ref{IvCurved} also has a multi-dimensional generalization, proved in the same way.
Consider an ellipsoid (\ref{ellmult}), and let $\Pi$ be an $n-1$-dimensional parallelepiped on it, bounded by confocal quadrics.

\begin{theorem} \label{IvCurvedmulti}
All pairs of opposite vertices of $\Pi$ are at equal geodesic distances.  
\end{theorem}  

Similarly, one can also prove multi-dimensional versions of the spherical and hyperbolic Ivory Lemmas, see \cite{SW04,Hor11}. We do not dwell on it here.

\section{Ivory's Lemma for Liouville and St\"ackel nets} \label{LSmetr}
\subsection{Liouville nets}
A Riemannian metric in a domain $U \subset \R^2$ is called a \emph{Liouville metric} if in coordinates $(q_1, q_2)$ of $\R^2$ it has the form
\begin{equation}
\label{eqn:Liou}
ds^2 = (u_1 - u_2)(v_1 dq_1^2 + v_2 dq_2^2)
\end{equation}
for some smooth functions $u_i = u_i(q_i)$, $v_i = v_i(q_i)$, $i = 1,2$.
% Particular cases of Liouville metrics are:
% \begin{itemize}
% \item Surfaces of revolution, in the standart parallel-meridian coordinates.
% \item Euclidean metric on $\R^2$ in the elliptic coordinates (see Example \ref{} below).
% \item Metric on a central quadric in $\R^3$ in the elliptic coordinates (see Example \ref{} below).
% \end{itemize}
Liouville gave an explicit form of the geodesics of the metric \eqref{eqn:Liou}. His work was a straightforward generalization of Jacobi's description of geodesics on an ellipsoid.

Note that any coordinate change $q_1 = f_1(q'_1)$, $q_2 = f_2(q'_2)$ transforms a Liouville metric element to a Liouville metric element. Any coordinate change for which coordinate lines remain coordinate lines is of that form. Therefore it is possible to speak of a \emph{Liouville net} instead of a Liouville metric.

It turns out that Liouville nets are characterized by the Ivory property.

\begin{theorem}[Blaschke-Zwirner]
\label{thm:LiouvIvory}
Liouville nets satisfy the Ivory property: the lengths of the geodesic diagonals in all net quadrilaterals are equal.
Conversely, if the lengths of the geodesic diagonals in all net quadrilaterals are equal, then the metric has a Liouville form in any coordinate system for which the net lines are coordinate lines.
\end{theorem}

The first part of Theorem \ref{thm:LiouvIvory} was proved by Zwirner \cite{Zw}, who gave credit to van der Waerden. Zwirner proved the second part under the assumption of analyticity of the metric. A different proof was given by Blaschke \cite{Bla28} (see also \cite[\S 56]{Bla50}), who at the same time generalized Theorem \ref{thm:LiouvIvory} to higher dimensions, see Section \ref{sec:Sta} below. Thimm in his PhD thesis \cite{Thimm78} gave a modern account of Blaschke's argument and filled a gap in the proof of the second part.

We reproduce below the Blaschke-Thimm's proof of the first part of Theorem \ref{thm:LiouvIvory}.

\subsubsection{Geodesics of Liouville metrics}
A coordinate change $dq'_i = \sqrt{v_i} dq_i$ transforms the metric element \eqref{eqn:Liou} to
\begin{equation}
\label{eqn:LiouSpec}
ds^2 = (u_1 - u_2)(dq_1^2 + dq_2^2).
\end{equation}
We will describe the geodesics of the metrics in the above form.

For every Riemannian metric there is the associated Hamiltonian system on $T^*U$ with the energy function equal to half  the square norm of a cotangent vector. Constant speed geodesics in $U$ are projections of the integral curves of the corresponding Hamiltonian flow.

\begin{lemma}
The Hamiltonian
$$
%\label{eqn:HamEllPlane}
H(q, p) = \frac12 \frac{p_1^2 + p_2^2}{u_1 - u_2}
$$
has a first integral
\[
f(q, p) = \frac12 \frac{u_2 p_1^2 + u_1 p_2^2}{u_1 - u_2}.
\]
\end{lemma}
\begin{proof}
The Poisson bracket $\{H, f\}$ vanishes, namely
\[
\frac{\partial H}{\partial p_1} \frac{\partial f}{\partial q_1} + \frac{\partial H}{\partial p_2} \frac{\partial f}{\partial q_2}= \frac{\partial H}{\partial q_1} \frac{\partial f}{\partial p_1}+\frac{\partial H}{\partial q_2} \frac{\partial f}{\partial p_2},
\]
as needed.
\end{proof}

\begin{theorem}[Liouville]
\label{thm:GeodParam}
The geodesics of the metric \eqref{eqn:LiouSpec} are given by
\[
\int \frac{dq_1}{\sqrt{u_1 - 2\alpha_2}} \pm \int \frac{dq_2}{\sqrt{2\alpha_2 -  u_2}} = \const.
\]
The unit speed parametrization is determined (up to a time shift and time reversal) by
\[
\int \frac{u_1\, dq_1}{\sqrt{u_1 - 2\alpha_2}} \pm \int \frac{u_2\, dq_2}{\sqrt{2\alpha_2 - u_2}} = t,
\]
where the  choice of the sign agrees with that in the first formula.
\end{theorem}

\begin{proof}
Solving the system
\[
H(q, p) = \frac12 \frac{p_1^2 + p_2^2}{u_1 - u_2} = \alpha_1, \quad f(q, p) = \frac12 \frac{u_2 p_1^2 + u_1 p_2^2}{u_1 - u_2} = \alpha_2,
\]
we obtain
\[
p_1 = \pm \sqrt{2(\alpha_1 u_1 - \alpha_2)}, \quad p_2 = \pm \sqrt{2(\alpha_2 - \alpha_1 u_2)}.
\]
For arbitrary choices of constants $\alpha_i$, these equations describe the lifts to $T^*M$ of geodesics of constant speed $2\alpha_1$.

To describe the geodesics explicitly, use the Hamilton-Jacobi approach. Namely, the equation
\[
H(q, \grad W) = \alpha_1
\]
can be integrated due to the above separation of variables:
\[
W(q, \alpha) = \pm \int \sqrt{2(\alpha_1 u_1 - \alpha_2)}\, dq_1 \pm \int \sqrt{2(\alpha_2 - \alpha_1 u_2)}\, dq_2.
\]
Then the relations
\[
\frac{\partial W}{\partial \alpha_2} = \const, \qquad \frac{\partial W}{\partial \alpha_1} = t
\]
yield  unparametrized equations of constant speed geodesics and their parameterizations, respectively. Remembering that $2\alpha_1$ is the speed, we obtain formulas stated in the theorem.
\end{proof}

\subsubsection{Ivory's lemma for Liouville nets}
\label{sec:LiouvIvory}
Consider the coordinate quadrilateral $[q_1^0, q_1^1] \times [q_2^0, q_2^1]$.
Let $\gamma$ be the unit speed parametrized diagonal from $(q_1^0, q_2^0)$ to $(q_1^1, q_2^1)$:
\[
\gamma(t) = (q_1(t), q_2(t)), \quad \gamma(t_0) = (q_1^0, q_2^0), \quad \gamma(t_1) = (q_1^1, q_2^1).
\]
Without loss of generality, assume that the equations in Theorem \ref{thm:GeodParam} for $\gamma$ have the plus sign between the integrals. Then, for every $t \in [t_0,t_1]$, we have
\[
\int\limits_{q_1^0}^{q_1(t)} \frac{dq_1}{\sqrt{u_1-2\alpha_2}} + \int\limits_{q_2^0}^{q_2(t)} \frac{dq_2}{\sqrt{2\alpha_2-u_2}} = 0,
\]
and the length of this diagonal equals
\[
t_1 - t_0 = \left| \int\limits_{q_1^0}^{q_1^1} \frac{u_1\, dq_1}{\sqrt{u_1-2\alpha_2}} + \int\limits_{q_2^0}^{q_2^1} \frac{u_2\, dq_2}{\sqrt{2\alpha_2-u_2}} \right|.
\]

Let $\overline{\gamma}$ be the unit speed geodesic with $\overline{\gamma}(t_0) = (q_1^0, q_2^1)$, with the same $\alpha_2$-value as $\gamma$, but with the minus sign between the integrals in Theorem \ref{thm:GeodParam}.
We claim that
\[
\overline{\gamma}(t_1) = (q_1^1, q_2^0),
\]
that is, $\gamma$ passes through the opposite corner of the quadrilateral, and its segment between the corners has the same length as the first diagonal.

Indeed, interchanging the integration limits in the second integral in the two equations above yields
\[
\int\limits_{q_1^0}^{q_1^1} \frac{dq_1}{\sqrt{u_1-2\alpha_2}} - \int\limits_{q_2^1}^{q_2^0} \frac{dq_2}{\sqrt{2\alpha_2-u_2}} = 0,
\]
which implies that $\overline{\gamma}$ passes through $(q_1^1, q_2^0)$, and
\[
t_1 - t_0 = \int\limits_{q_1^0}^{q_1^1} \frac{u_1\, dq_1}{\sqrt{u_1-2\alpha_2}} - \int\limits_{q_2^1}^{q_2^0} \frac{u_2\, dq_2}{\sqrt{2\alpha_2-u_2}},
\]
which implies, in its turn, that $\overline{\gamma}$ attains $(q_1^0, q_2^1)$ at the time $t = t_1$. This completes the proof of the first part of Theorem \ref{thm:LiouvIvory}.

\subsubsection{Ivory implies Liouville}
The second part of Theorem \ref{thm:LiouvIvory} states that every Ivory net is a Liouville net.
We give only an idea of the proof of this statement.

First, it follows easily from consideration of infinitesimally thin net quadrilaterals that an Ivory net is orthogonal. Second, one can show that the unit tangent vectors of the diagonals in a net quadrilateral have equal or opposite covariant components at the corresponding points. (Here Blaschke's argument contains a gap filled by Thimm.) This implies that the geodesic flow allows separation of variables with respect to the net. Finally, separation of variables, together with orthogonality of the net, implies that the metric has the Liouville form.

\subsection{Higher dimensions: St\"ackel metrics}
\label{sec:Sta}
\subsubsection{St\"ackel's metric element}
St\"ackel introduced in \cite{Sta93} the following class of metric tensors.
\begin{definition}
A Riemannian metric $ds^2$ on a domain $U \subset \R^n$ is called \emph{St\"ackel metric} if there is a $\GL(n)$-valued function
\[
M \colon U \to \GL(n), \quad M(q) = \big(u_{ij}(q_i)\big)_{i,j=1}^n
\]
with the $i$-th row depending only on the $i$-th coordinate of $\R^n$, and such that
\begin{equation}
\label{eqn:StaMet}
ds^2 = \sum_{i=1}^n g_{ii}(q)\, dq_i^2, \quad g_{ii}(q) = (-1)^{1+i} \frac{\det M}{\det M_{i1}},
\end{equation}
where $M_{i1}$ is the minor obtained by deleting from $M$ the $i$-th row and the first column.

In particular, the coordinate vector fields of $\R^n$ are pairwise orthogonal with respect to a St\"ackel metric.
\end{definition}

For $n=2$, St\"ackel metrics are exactly those of Liouville: the metric element \eqref{eqn:Liou} corresponds to $M = \begin{pmatrix} u_1 v_1 & v_1\\ -u_2 v_2 & -v_2 \end{pmatrix}$.

A \emph{St\"ackel net} is the net of coordinate hyperplanes of a St\"ackel metric. A net determines the coordinates only up to coordinate change $q_i = f_i(q'_i)$ but, as in the Liouville case, this transforms a St\"ackel metric element to a St\"ackel metric element.

\subsubsection{Separation of variables and the Ivory property}
\label{sec:StaeckelSeparation}
The kinetic energy Hamiltonian on $T^*U$ corresponding to \eqref{eqn:StaMet} has the form
\begin{equation}
\label{eqn:HamSta}
H(q,p) = \frac12 \sum_{i=1}^n (-1)^{1+i} \frac{\det M_{i1}}{\det M}\, p_i^2.
\end{equation}

By generalizing the Jacobi-Liouville method, St\"ackel has proved the following theorem.

% (All Hamiltonians that can be integrated by means of separation of variables were described by Iarov-Iarovoi in \cite{II63}.)

\begin{theorem}[St\"ackel]
Let $U \subset \R^n$ be an open domain and
\[
ds^2 = \sum_{i=1}^n g_{ii}(q)\, dq_i^2
\]
be a Riemannian metric on $U$ with pairwise orthogonal coordinate vector fields. Then the Hamilton-Jacobi equation for the associated kinetic energy on the cotangent bundle $T^*U$ can be completely solved through separation of variables if and only if $ds^2$ is of the form \eqref{eqn:StaMet}.
\end{theorem}

First, one shows that the functions $\alpha_2, \ldots, \alpha_n$ defined by
\[
\begin{pmatrix} \alpha_1\\ \vdots\\ \alpha_n \end{pmatrix} = \frac12 M^{-1} \begin{pmatrix} p_1^2\\ \vdots\\ p_n^2 \end{pmatrix}
\]
are pairwise commuting first integrals of the Hamiltonian (note that $\alpha_1 = H$). This implies that the constant speed geodesics satisfy the system of equations
\[
\begin{pmatrix} p_1^2\\ \vdots\\ p_n^2 \end{pmatrix} = 2M \begin{pmatrix} \alpha_1\\ \vdots\\ \alpha_n \end{pmatrix}.
\]
As the $i$-th row of matrix $M$ depends only on $q_i$, we have
\[
p_i^2 = h_i(q_i, \alpha), \quad \alpha = (\alpha_1, \ldots, \alpha_n),
\]
which allows to integrate the equation $H(q, \grad W) = \alpha_1$:
\begin{equation}
\label{eqn:WSt}
W(q, \alpha) = \sum_{i=1}^n \int \pm \sqrt{h_i(q_i, \alpha)} \, dq_i,
\end{equation}
and to obtain equations of the geodesics:
\begin{align}
\int \frac{u_{1i} dq_1}{\sqrt{h_1(q_1, \alpha)}} \pm \cdots \pm \int \frac{u_{ni} dq_n}{\sqrt{h_n(q_n, \alpha)}} &= \const, \quad i = 2, \ldots, n, \label{eqn:StNonparam}\\
\int \frac{u_{11} dq_1}{\sqrt{h_1(q_1, \alpha)}} \pm \cdots \pm \int \frac{u_{n1} dq_n}{\sqrt{h_n(q_n, \alpha)}} &= t, \label{eqn:StTime}
\end{align}
where $\alpha_1$ is set to $\frac12$ and, for each $j$, the $j$-th summands in all equations have the same sign.

% For a fixed $\alpha$, different choices of the $\pm$ signs under the integrals correspond to $2^{n-1}$ families of geodesics.

Take a coordinate parallelepiped on a Riemannian manifold with a St\"ackel metric element.
Exactly as in Section \ref{sec:LiouvIvory}, one can show that its $2^{n-1}$ great geodesic diagonals correspond to the same value of $\alpha$ but to different choices of the $\pm$ signs in the integrals \eqref{eqn:StNonparam}, \eqref{eqn:StTime}, and that all these diagonals have the same length.

\begin{theorem}[Blaschke]
St\"ackel metrics possess the Ivory property: in every parallelepiped bounded by the coordinate hypersurfaces all great geodesic diagonals have equal lengths. Vice versa, if coordinate parallelepipeds have diagonals of equal lengths, then the metric has a St\"ackel form, after possibly an independent coordinate change.
\end{theorem}

\subsection{Geometric properties of St\"ackel metrics}
St\"ackel nets share many properties with confocal quadrics.

\subsubsection{St\"ackel nets induced on coordinate hypersurfaces}
Blaschke \cite{Bla28} proved the following theorem.

\begin{theorem}[Blaschke]
\label{thm:CoordStaeckel}
The restriction of a St\"ackel metric element to any coordinate hypersurface $q_i = \const$ is again a St\"ackel metric element. That is, a St\"ackel net induces on all of its coordinate hypersurfaces St\"ackel nets of one dimension lower.
\end{theorem}
By induction, the same is true for intersection of any number of coordinate hypersurfaces.

In the submanifold $q_1 = c_1, \ldots, q_k = c_k$ of an $n$-dimensional St\"ackel net there are two ways of measuring distances: extrinsically and intrinsically. Similarly to Remark \ref{rem:ExtrIntrDiag}, we have the following.

\begin{corollary}
Let $Q$ be a coordinate parallelepiped of a St\"ackel metric, and let $F$ be a face of $Q$. Then all great intrinsic diagonals of $F$ have the same length, and all great extrinsic diagonals of $F$ have the same length as well.
\end{corollary}
\begin{proof}
The intrinsic diagonals are equal because the metric in the coordinate subspace spanned by $F$ is St\"ackel. The extrinsic diagonals are equal because $F$ can be viewed as a limit of $n$-dimensional parallelepipeds, and the great diagonals of these parallelepipeds converge to the extrinsic diagonals of $F$.
\end{proof}

\subsubsection{Examples}
The simplest examples of Liouville metrics are surfaces of revolution. The Ivory lemma holds for them trivially by symmetry reasons. Clairaut's theorem and the resulting equations of geodesics can be viewed as a special case of integration of geodesics on Liouville surfaces.

\begin{example}[Elliptic coordinates in $\R^2$]
In the elliptic coordinates associated with the ellipse
\[
\left\{ \frac{x^2}{a} + \frac{y^2}{b} = 1 \right\}, \quad a > b > 0,
\]
the Euclidean metric in $\R^2$ has the form
\[
dx^2 + dy^2 = (\lambda - \mu) \left( -\frac{d\lambda^2}{4(a-\lambda)(b-\lambda)} + \frac{d\mu^2}{4(a-\mu)(b-\mu)} \right),
\]
which is Liouville.
\end{example}

\begin{example}[Ellipsoidal coordinates in $\R^3$]
\label{exl:Ell3}
The Euclidean metric in $\R^3$ has in the ellipsoidal coordinates the form
\begin{multline*}
dx^2 + dy^2 + dz^2 =  \frac{(\lambda-\mu)(\lambda-\nu)}{4(a-\lambda)(b-\lambda)(c-\lambda)} d\lambda^2  \\
- \frac{(\lambda-\mu)(\mu-\nu)}{4(a-\mu)(b-\mu)(c-\mu)} d\mu^2 + \frac{(\lambda-\nu)(\mu-\nu)}{4(a-\nu)(b-\nu)(c-\nu)} d\nu^2.
\end{multline*}
Here $a > \lambda > b > \mu > c > \nu$. This is a St\"ackel metric \eqref{eqn:StaMet} with
\[
M =
\begin{pmatrix}
\frac{\lambda^2}{h(\lambda)} & \frac{\lambda}{h(\lambda)} & \frac{1}{h(\lambda)}\\
\frac{\mu^2}{h(\mu)} & \frac{\mu}{h(\mu)} & \frac{1}{h(\mu)}\\
\frac{\nu^2}{h(\nu)} & \frac{\nu}{h(\nu)} & \frac{1}{h(\nu)}
\end{pmatrix},
\]
where $h(\lambda) = 4(a-\lambda)(b-\lambda)(c-\lambda)$.
\end{example}

\begin{example}[Sphero-conical coordinates in $\R^3$]
The sphero-conical coordinates are $(r,\lambda,\mu)$, where $r^2 = x^2 + y^2 + z^2$, and $\lambda$ and $\mu$ are determined by equation \eqref{cone}, $a > \lambda > b > \mu > c$. The Euclidean metric has the form
\[
dr^2 + r^2(\lambda - \mu) \left( \frac{d\lambda^2}{4(a-\lambda)(b-\lambda)(c-\lambda)} - \frac{d\mu^2}{4(a-\mu)(b-\mu)(c-\mu)} \right),
\]
which is St\"ackel for
\[
M =
\begin{pmatrix}
1 & -\frac{1}{r^2} & 0\\
0 & \frac{\lambda}{h(\lambda)} & \frac{1}{h(\lambda)}\\
0 & \frac{\mu}{h(\mu)} & \frac{1}{h(\mu)}
\end{pmatrix}.
\]
\end{example}

Any St\"ackel net in $\R^n$ (with the Euclidean metric) consists of quadrics (and, at least for $n=2,3$, forms a confocal system or one of its degenerations). This is a result of Weihnacht \cite{Wei24}, reproved in a simpler way by Blaschke in \cite{Bla28} using the Ivory property of the St\"ackel metrics.

\begin{example}[Ellipsoid and sphere]
The ellipsoid $\frac{x^2}{a} + \frac{y^2}{b} + \frac{z^2}{c} = 1$ can be viewed as the level surface $\nu = 0$ in the ellipsoidal coordinates, Example \ref{exl:Ell3}. Thus its intrinsic metric is given in the $(\lambda, \mu)$-coordinates by the formula
\[
ds^2 = (\lambda-\mu) \left( \frac{\lambda\, d\lambda^2}{4(a-\lambda)(b-\lambda)(c-\lambda)} - \frac{\mu\, d\mu^2}{4(a-\mu)(b-\mu)(c-\mu)} \right),
\]
which is of Liouville form.

Similarly, the unit sphere $r=1$ has in the conical coordinates the metric
\[
ds^2 = (\lambda - \mu) \left( \frac{d\lambda^2}{4(a-\lambda)(b-\lambda)(c-\lambda)} - \frac{d\mu^2}{4(a-\mu)(b-\mu)(c-\mu)} \right),
\]
which is also of Liouville form.
\end{example}

\begin{example}[Intersections of confocal quadrics]
Confocal quadrics form a St\"ackel net with respect to the Euclidean metric in any dimension. Therefore, by Theorem \ref{thm:CoordStaeckel}, the intersection of any number of confocal quadrics is a Riemannian manifold with a St\"ackel net. For example, the intersection of two confocal quadrics in $\R^4$ carries a Liouville net.
\end{example}

Theorem \ref{thm:CoordStaeckel} allows to derive the Ivory lemma on the sphere and in the hyperbolic plane from the Ivory lemma in the Euclidean space. Indeed, the Euclidean metric is St\"ackel with respect to the conical coordinates, and the unit sphere is a coordinate hypersurface of the conical coordinate system. The other coordinate hypersurfaces (quadratic cones) intersect the sphere along the ``spherical conics'', see Section \ref{spherehyp}. Thus the great diagonals of the parallelepipeds cut out by spherical conics have the same length in the spherical metric.

The same argument works for the hyperbolic space, realized as a component of a two-sheeted hyperboloid in $\R^{n+1}$ equipped with the Lorentzian metric tensor $-dq_0^2 + dq_1^2 + \cdots + dq_n^2$. The Lorentzian metric has a St\"ackel form with respect to the appropriate analog of the conical coordinate system. It follows that the induced metric on the two-sheeted hyperboloid is also St\"ackel, and hence has the Ivory property.

\subsubsection{Families of geodesics}
\label{sec:GeodFam}
Call a family of geodesics \eqref{eqn:StNonparam} with fixed values of $\alpha_2, \ldots, \alpha_n$ an \emph{$\alpha$-family}. (Recall that $2\alpha_1$ is the speed, so we fix $\alpha_1 = \frac12$.) Locally, the geodesics in an $\alpha$-family are split into $2^{n-1}$ subfamilies (\emph{signed $\alpha$-families}) corresponding to different choices of the $\pm$-signs; each subfamily is parametrized by $n-1$ parameters, the constants on the right hand side of equations \eqref{eqn:StNonparam}.

\begin{lemma}
Locally, each signed $\alpha$-family is orthogonal to the same family of hypersurfaces.
\end{lemma}
\begin{proof}
Indeed, a signed $\alpha$-family is the gradient flow of the function (\ref{eqn:WSt}). Hence, the level hypersurfaces of $W$ are orthogonal to all curves of this family.
\end{proof}

However, a geodesic can contain arcs from different signed subfamilies of the same $\alpha$-family. Sign changes occur at the points of tangency with the coordinate hypersurfaces. Note that in \eqref{eqn:StNonparam} and \eqref{eqn:StTime} we have $h_i(q_i(t), \alpha) > 0$. Now, if $h_i(a_i, \alpha) = 0$ and $q_i(t)$ approaches $a_i$, then there are two possibilities:
\begin{itemize}
\item $q_i(t)$ tends to $a_i$ as $t$ tends to infinity. The geodesic is asymptotic to the coordinate hypersurface $q_i = a_i$;
\item $q_i(t)$ attains $a_i$ in finite time and then ``bounces back''. The geodesic is tangent to the coordinate hypersurface $q_i = a_i$; the signs of $\sqrt{h_i(q_i,\alpha)}$ change in all equations.
\end{itemize}

In the latter case all geodesics which are close to the chosen one and belong to the same $\alpha$-family are tangent to the same coordinate hypersurfaces.

In particular, on Liouville surfaces, $\alpha$-families are tangent to the coordinate curves, and hence orthogonal to their involutes. Thus these involutes are the level curves of the function \eqref{eqn:WSt}.

\subsubsection{Billiard integrability, Ivory, and Poncelet}
\label{sec:StBill}
Consider the billiard inside a coordinate parallelepiped of a St\"ackel net. Equations \eqref{eqn:StNonparam} and the diagonality of the St\"ackel metric imply that a reflection in a coordinate hypersurface $q_i = a_i$ preserves all integrals $\alpha_1, \ldots, \alpha_n$, and changes only the sign of $\sqrt{h_i(q_i,\alpha)}$. This is an analog of the billiard property of confocal quadrics: the incoming and the outgoing rays are tangent to the same $n-1$ confocal quadrics.

As a result, a billiard trajectory inside a St\"ackel parallelepiped behaves similarly to that inside a parallelepiped bounded by confocal quadrics, compare with Section \ref{multidim}. Let us repeat the arguments in a slightly different way.

The reflection in a coordinate hypersurface transforms the constants on the right hand side of \eqref{eqn:StNonparam} linearly, and this linear transformation depends only on $\alpha$. It follows that if the billiard inside a domain bounded by coordinate hypersurfaces has a periodic trajectory, then all nearby trajectories in the same $\alpha$-family are periodic.

Note that, for a trajectory to be periodic, it is necessary that it reflects an even number of times from each coordinate hypersurface (including tangency to coordinate hypersurfaces, which can be viewed as a limit case of reflection). It follows from \eqref{eqn:StTime} that all periodic trajectories from the same $\alpha$-family have the same length.

In particular, this implies the periodicity and constant length for trajectories inside a coordinate parallelepiped ``parallel'' to a big diagonal. This also implies the Poncelet theorem for St\"ackel metrics: here the billiard table is bounded by $q_i=a$ and $q_i=b$ for the same $i$, and the trajectory is chosen to be tangent to one of these hypersurfaces (by the results in the previous section, all geodesics from the same family are tangent to the same hypersurface).

\subsubsection{Graves' theorem and Staude's construction}
Darboux \cite[Livre VI, Chapitre I]{Da1} proved that the Graves theorem (see Figure \ref{string}) holds also for coordinate curves of Liouville metrics. The argument is implicit in our Section \ref{sec:GeodFam}.

A higher-dimensional analog of the Graves theorem is the string construction of confocal quadrics described by Staude \cite{Sta82}. Blaschke \cite{Bla28} used his representation of geodesics of a St\"ackel metric to show that Staude's result holds for all St\"ackel metrics. Usually one cites a very elegant special case of Staude's string construction which involves the focal ellipse and focal hyperbola, see \cite[\S I.4]{HC52} for an illustration.

In the general Staude's construction in dimension $3$, the string first wraps along the intersection curve of two confocal quadrics $F_1$ and $F_2$, and then along the quadric $F_2$. Pulling the string tight at a point, this point will describe a confocal quadric ``parallel'' to $F_1$. (More generally, for St\"ackel nets, the string first wraps along a coordinate line and then along a coordinate surface.)

% Take coordinate hypersurfaces $q_1 = a_1$ and $q_2 = a_2$ and two points $A$ and $B$ sufficiently close on their intersection curve. Assume that near $A$ and $B$ each of the domains $q_1 \le a_1$, $q_2 \le a_2$ is convex. Take a point $C$ in the common complement of these domains such that the shortest curves from $A$ to $C$ and from $B$ to $C$ first go along the intersection curve $q_1 = a_1, q_2 = a_2$, then along intrinsic geodesics of the hypersurface $q_1 = a_1$ and then inside $q_1 > a_1, q_2 > a_2$ (imagine the point $C$ close to the line $q_1=a_1, q_2=a_2$ between $A$ and $B$ and closer to $q_1=a_1$ than to $q_2=a_2$). Then as the point $C$ moves while preserving the total length of these two shortest curves, it stays on a coordinate hypersurface $q_2=b_2$ for some $b_2$.
% 
% {\bf Is it possible to add a figure illustrating the above paragraph?}

% It follows that the unit speed geodesics lift to the curves in $T^*M$ which are solutions to the system
% \[
% \frac{p_1^2 + p_2^2}{u_1 - u_2} = 1, \quad \frac{u_2 p_1^2 + u_1 p_2^2}{u_1 - u_2} = \alpha
% \]
% for some constant $\alpha$. This yields
% \[
% p_1^2 = u_1 - \alpha, \quad p_2^2 = \alpha - u_2
% \]

\section{Newton and Ivory theorems in spaces of constant curvature} \label{NIcurv}

%\subsection{Newton and Ivory theorems in the Euclidean space}
%In his ``Mathematical Principles of Natural Philosophy'' Newton proved the following theorem.
%\begin{theorem}
%The gravitational field created by a spherical shell equals zero in the region bounded by the shell. In the exterior region the field is the same as the one created by the total mass of the shell concentrated at its center.
%\end{theorem}
%Newton's theorem can be generalized to spheres in $\R^n$, if the gravitational force is defined to be inversely proportional to the $(n-1)$-st power of the distance.

%Laplace and Ivory proved an analog of Newton's theorem for ellipsoids.
%\begin{definition}
%A \emph{homeoid} is the domain bounded by two homothetic ellipsoids with a common center.
%\end{definition}

%\begin{theorem}
%The graviational field created by a homeoidal shell equals zero in the region bounded by the shell. If the shell is infinitely thin, then the equipotential surfaces exterior to it are the confocal ellipsoids.
%\end{theorem}

%{\bf Why is it important for the shell to be infinitely, and not finitely, thin?}

%The equipotentiality of the confocal ellipsoids was proved by Laplace through direct computations. Ivory used a geometric argument based on the lemma that carries his name.

%In the next sections we will formulate analogs of Newton's and Ivory's theorems on the sphere and in the hyperbolic space of any dimension. Our proofs will be similar to those of Newton and Ivory for the Euclidean case.

\subsection{The gravitational potential on the sphere and in the hyperbolic space}
\label{sec:GravPot}

First of all, we need to describe the law of attraction on the sphere and in the hyperbolic space.

\begin{definition}
The attraction between two points at distance $r$ in $\Sph^n$, respectively in $\HH^n$, is inversely proportional to $\sin^{n-1} r$, respectively to $\sinh^{n-1} r$, and is directed along the geodesic (shortest, in the spherical case) connecting these points.
\end{definition}

As the following lemma shows, this law is the only one for which the force field of a point is divergence-free and rotationally invariant or, equivalently, the potential of a point mass is harmonic and rotationally invariant.

\begin{lemma}
\label{lem:GravSphHyp}
Every rotationally symmetric harmonic function on the sphere or in the hyperbolic space is, up to a constant factor, equal to
\begin{equation}
\label{eqn:FundSol}
u(r) =
\begin{cases}
\int_r^{\frac{\pi}2} \frac{dx}{\sin^{n-1}x} &\text{ in } \Sph^n,\\
\int_r^\infty \frac{dx}{\sinh^{n-1}x} &\text{ in }\HH^n,
\end{cases}
\end{equation}
where $r$ denotes the distance from the point mass.

\end{lemma}
\begin{proof}[First proof]
The gradient of a harmonic function is a divergence-free vector field. Hence the flux of $\nabla u$ through a sphere centered at the point mass is independent of the radius of the sphere. Since $\nabla u$ is orthogonal to the sphere and has a constant norm over the sphere, it follows that $\|\nabla u(r)\|$ is inversely proportional to the area of the sphere of radius $r$. The latter is proportional to $\sin^{n-1}r$, in the spherical, and to $\sinh^{n-1}r$, in the hyperbolic case. Integration in the radial directions produces the desired formulas.
\end{proof}

\begin{proof}[Second proof]
The spherical and the hyperbolic metrics have, in the polar coordinates, the form
\[
g =
\begin{cases}
dr^2 + \sin^2 r \cdot h &\text{ in }\Sph^n,\\
dr^2 + \sinh^2 r \cdot h &\text{ in }\HH^n,
\end{cases}
\]
where $h$ is the metric tensor on $\Sph^{n-1}$. It is a general fact that, with respect to a warped product metric $dr^2 + \phi^2(r) h$,
the $(1,1)$-Hessian of a function, that depends only on $r$, equals
\[
\Hess u(r) = u'' \cdot d\rho + u' \frac{\phi'}{\phi} \cdot d\sigma
\]
(here $d\rho$ is the ``vertical component'', that is, the $r$-component of a vector, and $d\sigma$ is the ``horizontal component''). Taking the trace, we obtain
\[
\Delta u = \Tr\Hess u = u'' + (n-1)u' \frac{\phi'}{\phi}.
\]
Hence a harmonic function, depending only on the $r$-coordinate, must satisfy
\[
u'' + (n-1)u \frac{\phi'}{\phi} = 0.
\]
This can be integrated to $\log u' = -(n-1) \log \phi + \const$, that is, $u'$ is a multiple of $\phi^{-(n-1)}$.

Applying this to $\phi(r) = \sin r$ or to $\sinh r$, we obtain formula \eqref{eqn:FundSol}.
\end{proof}

\begin{example}
For $n=3$, the integral in \eqref{eqn:FundSol} can be computed explicitly: $u(r) = \cot r$ and $u(r) = \coth r$, respectively.
\end{example}

A body $D \subset \Sph^n$ with a continuous mass density $\rho$ exerts at a point $p$ the gravitational potential $\int_D u(\|x-p\|)\, \rho(x) dx$, and similarly for $\HH^n$.
Instead of a gravitational potential, sometimes it is convenient to speak about an electrostatic potential. In particular, a negative mass can be interpreted as a negative charge.

\begin{remark}
In the spherical case, the potential \eqref{eqn:FundSol} satisfies $u(\pi-r) = -u(r)$. Thus a negative unit charge at the south pole has the same effect as a positive unit charge at the north pole. Also, any electrostatic potential on the sphere is antisymmetric, and any charge distribution is equivalent to a distribution with the support in a hemisphere.
\end{remark}

\subsection{Newton's theorem}
\label{sec:Newton}
Figure \ref{fig:NewtonToIvory}, left, illustrates a proof of Newton's theorem in the Euclidean case.
Take a point $p$ inside a spherical shell and consider a thin two-sided cone with apex $p$. The intersection of the cone with the shell consists of two opposite truncated cones. If we show that the forces exerted on $p$ by these two components compensate each other, then Newton's theorem will follow. The norm of the gravitational field at distance $r$ from $p$ is inversely proportional to $r^{n-1}$, and the width of the cone with apex at $p$ is proportional to $r^{n-1}$. Therefore the force exerted by a thin truncated cone is proportional to its height. Since every line through $p$ intersects the shell in two segments of equal length, the forces exerted by opposite truncated cones compensate each other.

In a formal way:
\[
\int_D \frac{v}{\|x-p\|^{n-1}}\, dx = \int_{\Sph^{n-1}} \int_{a(v)}^{b(v)} v\, dr\, d\phi = \int_{\Sph^{n-1}} (b(v) - a(v))v\, dr\, d\phi = 0
\]
because $b(-v) - a(-v) = b(v) - a(v)$.
Here $v = \frac{x-p}{\|x-p\|}$, and $[a(v),b(v)]$ is the interval of intersection of the ray in direction $v$ with the shell $D$.

The form of the gravitational field outside of the shell follows by symmetry arguments.

The following is an analog of Newton's theorem for spaces of constant curvature. In the case of $3$-dimensional sphere, it was proved by V. Kozlov \cite{Koz00} by way of a direct computation.
\begin{theorem}
\label{thm:NewtonSph}
The gravitational field created by a spherical shell in the hyperbolic space equals zero in the region bounded by the shell. Outside of the shell, it is the same as the field created by the total mass of the shell concentrated at the center.

In the spherical space, a spherical shell exerts no force in the smaller of the two regions bounded by it, as well as inside the antipodal region. Between the sphere and its antipode, the field is the same as the one created by the total mass of the spherical shell concentrated at the center.
\end{theorem}
\begin{proof}
The vanishing of the field inside the shell (and its antipode, in the spherical case) can be proved by  Newton's argument. Again, there are two main points. First, every line through an interior point intersects the shell in two segments of equal lengths. Second, the norm of the field ($\sinh^{-(n-1)} r$, respectively $\sin^{-(n-1)} r$) cancels the proportionality factor of the area element of the sphere at distance~$r$.

The force field outside the shell is a rotationally symmetric divergence free vector field. Hence, by Lemma \ref{lem:GravSphHyp}, it is proportional to the field created by a certain mass concentrated at the center. In the hyperbolic case, the asymptotics of the force field at infinity implies that this mass equals the total mass of the spherical shell.

In the spherical case we cannot use the asymptotics in order to determine this mass. Instead, it can be computed by integrating the fundamental solution \eqref{eqn:FundSol} over the mass distribution, as it is done in \cite{Koz00}. Alternatively, since we know the result of this integration in the hyperbolic case (the mass of the point equals to the mass of the shell), the same result holds in the spherical case, because the only difference between the integrals is in substituting $\sin$ for $\sinh$.
\end{proof}

\begin{figure}[ht]
\centering
\includegraphics{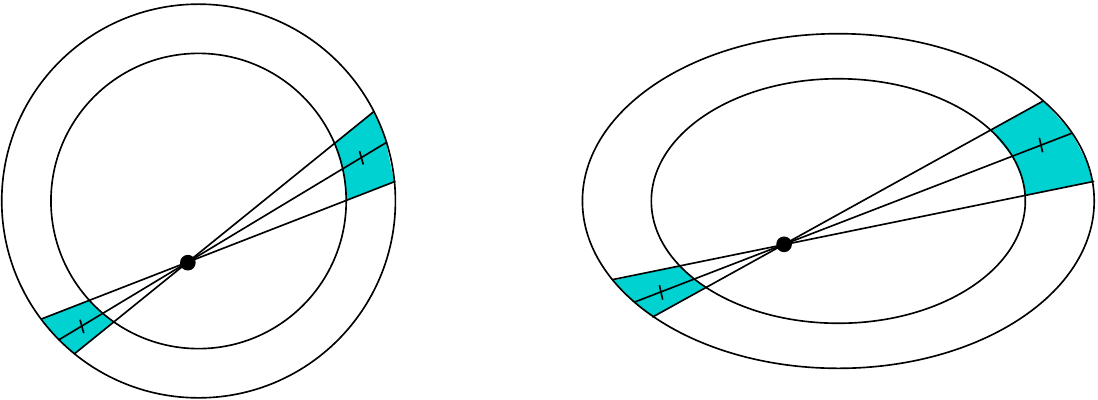}
\caption{Vanishing of the force field inside spherical and homeoidal shells.}
\label{fig:NewtonToIvory}
\end{figure}

\subsection{Homeoids and homeoidal densities}
The argument from Section \ref{sec:Newton} also proves the vanishing of the gravitational force inside a homeoidal shell: since a homeoidal shell in $\R^n$ is an affine image of a spherical shell, each line through an interior point $p$ intersects it in two segments of equal lengths, see Figure \ref{fig:NewtonToIvory}, right. For hyperbolic and spherical analogs of the Ivory theorem we have to find non-spherical shells in $\HH^n$ and $\Sph^n$ with the same property.

\begin{definition}
Let $C$ be an elliptic cone in $\R^{n+1}$, that is the zero set of a quadratic form of index $1$:
\[
C = \{x \in \R^{n+1} \mid q(x) = 0\}, \quad \sign q = (-, +, \ldots, +).
\]
The cone $C$ intersects the unit sphere $\Sph^n = \{\|x\| = 1\} \subset \R^{n+1}$ along two diametrically opposite components:
\[
C \cap \Sph^n = E \cup -E.
\]
Each component is called a \emph{spherical ellipsoid}.

If an elliptic cone $C$ is contained in the standard light-cone $\{x \in \R^{n+1} \mid \|x\|^2_{n,1} < 0\}$ (where $\|x\|^2_{n,1} = -x_0^2 + x_1^2 + \cdots + x_n^2$), then its intersection with the upper sheet $\HH^n$ of the hyperboloid $\|x\|^2_{n,1} = -1$ is called a \emph{hyperbolic ellipsoid}.
\end{definition}

\begin{definition}
Let $q$ be a quadratic form defining a spherical or hyperbolic ellipsoid $E$. The shell between two level sets $\{\epsilon_1 \le q(x) \le \epsilon_2\}$ of $q$, intersected with $\Sph^n$ or $\HH^n$, is called a \emph{homeoid} with the core $E$.
The level sets are assumed to lie sufficiently close to $E$, so that the shell is homeomorphic to a cylinder over $E$. On the other hand,
we allow the numbers $\epsilon_1$ and $\epsilon_2$ to be of the same sign, so that the core $E$ may lie outside of the shell.
\end{definition}

% \begin{figure}[ht]
% \centering
% \includegraphics{SpherHomeoid.pdf}
% \caption{A spherical homeoid.}
% \end{figure}

\begin{lemma}
\label{lem:EqualHeights}
If a geodesic intersects a spherical or hyperbolic homeoid in two segments, then these two segments have equal lengths.
\end{lemma}
\begin{proof}
A geodesic is an intersection of $\Sph^n$ or $\HH^n$ with a two-dimensional vector space $L$. The segments inside a homeoidal shell are circular or hyperbolic arcs enclosed between two level sets of the quadratic form $q$ restricted to $L$. The restriction has one of the signatures $(+,+)$, $(+,0)$, or $(+,-)$ (the first two possibilities can occur when both boundaries of the homeoid are exterior to the core). Possible views of the plane $L$ in the spherical case are depicted on Figure \ref{fig:EqualHeights}. The thick arcs have equal lengths by symmetry reasons.

\begin{figure}[ht]
\centering
\includegraphics[width=.9\textwidth]{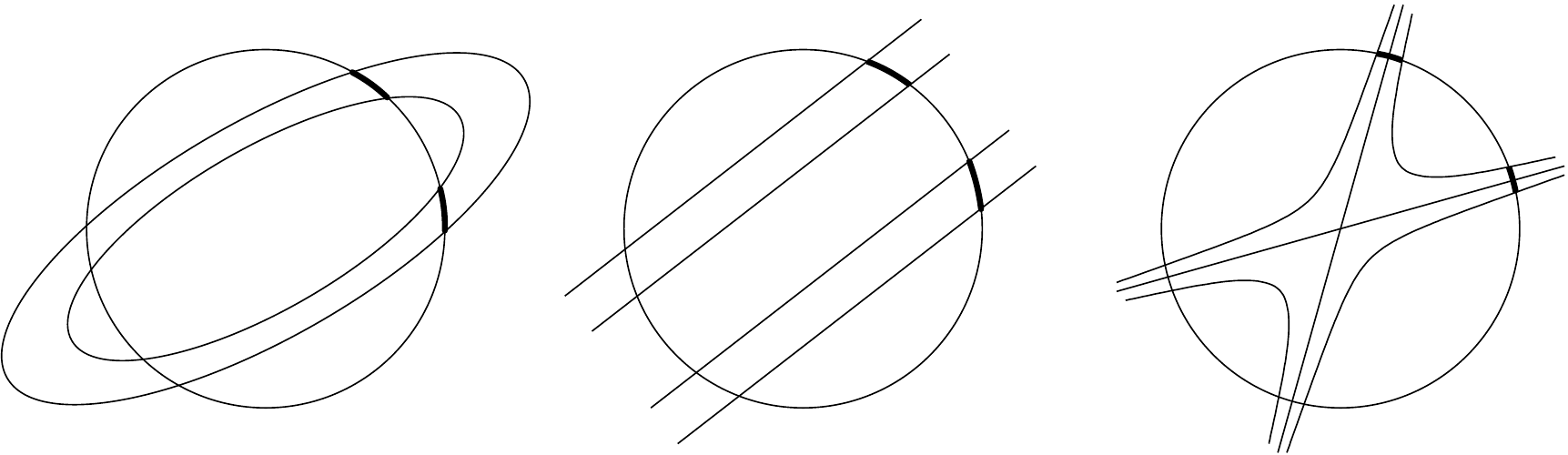}
\caption{Geodesics intersecting a spherical homeoid.}
\label{fig:EqualHeights}
\end{figure}

In the hyperbolic case, the pictures are not necessarily symmetric, but can be made symmetric by applying a hyperbolic isometry (that does not change hyperbolic lengths of the segments). Figure \ref{fig:EqualHeightsHyp} illustrates the case when the restriction of $q$ to $L$ has signature $(+,0)$.

Alternatively, the hyperbolic case can be proved by a direct computation, see Lemma \ref{lem:ArnSph}.
\end{proof}

\begin{figure}[ht]
\centering
\includegraphics[width=.8\textwidth]{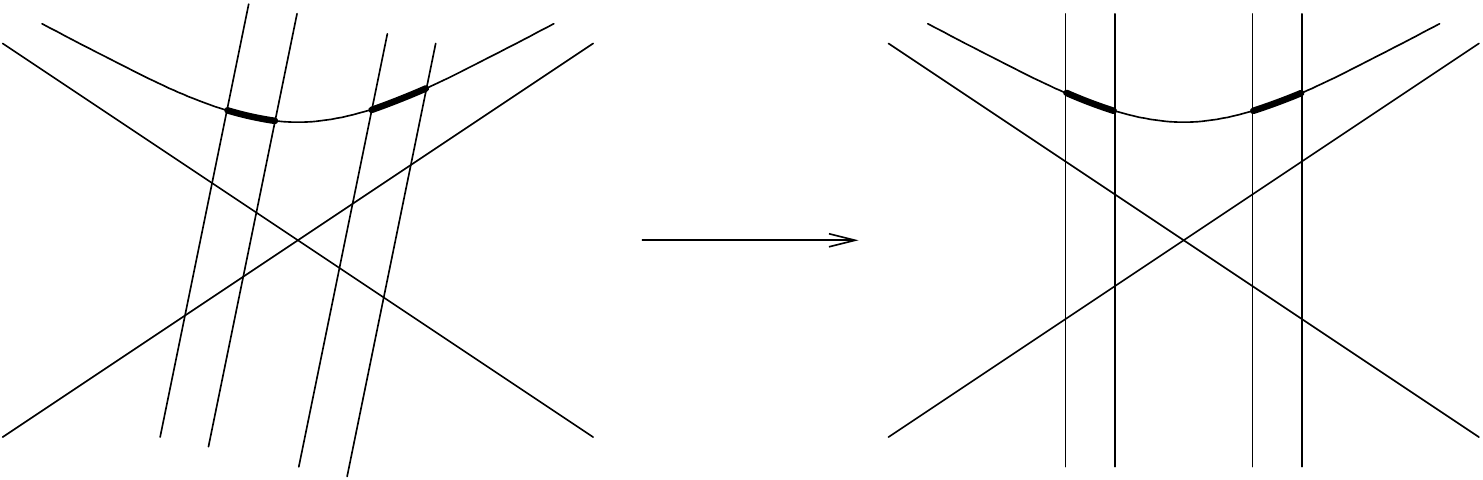}
\caption{Geodesic intersecting a hyperbolic homeoid: applying a hyperbolic isometry.}
\label{fig:EqualHeightsHyp}
\end{figure}

Up to now we have discussed the potentials and fields created by uniformly dense full-dimensional objects. When a shell becomes infinitely thin, we can view it as a hypersurface equipped with a variable (mass or charge) density. The field created by a charged hypersurface has a discontinuity along the hypersurface. The difference between the one-sided limits is a vector field that is orthogonal to the hypersurface (provided that the charge is H\"older continuous) and has the norm proportional to the charge density, see \cite{Kel67}.

The \emph{homeoidal density} is the renormalized limit of the thickness of a homeoidal shell.

\begin{lemma}
\label{lem:HomDens}
A homeoidal density on a spherical ellipsoid is inversely proportional to $\|\grad q\|$, where $q$ is a quadratic form defining the ellipsoid.

A homeoidal density on an ellipsoid in the hyperbolic space is inversely proportional to the Minkowski norm of the gradient $\|\grad q\|_{n,1}$.
\end{lemma}
\begin{proof}
The distance between level sets of a function is inversely proportional, up to terms of higher order, to the norm of the gradient of the function. Hence the spherical homeoidal density is inversely proportional to the gradient of the restriction of $q$ to $\Sph^n$. Since the cone $\{q(x) = 0\}$ intersects $\Sph^n$ orthogonally, the gradient of $q$ and the gradient of $q|_{\Sph^n}$ coincide.

In the hyperbolic case the argument is the same, except that the orthogonality should be understood with respect to the Minkowski scalar product. In particular, the gradient of a function $q$ in the Minkowski space has the coordinates $\left( -\frac{\partial q}{\partial x_0}, \frac{\partial q}{\partial x_1}, \ldots, \frac{\partial q}{\partial x_n} \right)$.
\end{proof}

\begin{remark}
A charge distributed with a homeoidal density creates zero potential in the interior of the ellipsoid; hence the electrostatic field is orthogonal to the surface of the ellipsoid. In other words, if the charged particles are allowed to move freely within the surface, a homeoidal density will put them in equilibrium. The same happens in a solid conductor: free charges inside a body bounded by a surface concentrate on this surface according to its equilibrium density, see \cite[Section VII.1]{Kel67}.
\end{remark}

\subsection{Ivory's theorem}
\label{sec:IvoryThm}
Let us consider the spherical case first.
Choose an orthonormal basis for the quadratic form $q$ and write it as
\begin{equation}
\label{eqn:qDiag}
q(x) = \frac{x_1^2}{a_1} + \cdots + \frac{x_n^2}{a_n} - \frac{x_0^2}{b} = 0, \quad a_1, \ldots, a_n, b > 0.
\end{equation}
Consider the ellipsoid in the upper hemisphere:
\[
E = q^{-1}(0) \cap \Sph^n_+,
\]
where $\Sph^n_+ = \Sph^n \cap \{x_0 > 0\}$. Assume that all $a_i$ are distinct; without loss of generality, $a_1 > a_2 > \cdots > a_n > 0$. Then the associated quadratic forms
\[
q_\lambda(x) = \frac{x_1^2}{a_1-\lambda} + \cdots + \frac{x_n^2}{a_n-\lambda} - \frac{x_0^2}{b+\lambda}, \quad a_1 > \lambda > -b, \lambda \ne a_i
\]
give rise to a confocal family of spherical quadrics
\[
E_\lambda = q_\lambda^{-1}(0) \cap \Sph^n_+
\]
that split into $n$ subfamilies.
% For $0 > \lambda > -b$, the surfaces $E_\lambda$ are ellipsoids that cover the space between $E$ and the equator $x_0 = 0$; for $a_n > \lambda > 0$ these are ellipsoids in the interior of $E$, tending to a doubly covered ellipsoid of one dimension less.
Similarly to the Euclidean case, every point on the $n$-sphere with $x_i \ne 0$ for all $i$ lies on $n$ confocal quadrics from different subfamilies; this gives rise to ellipsoidal coordinates on the sphere. At the same time, the cones $q_\lambda^{-1} (0)$, together with the  spheres centered at the origin, can be viewed as a degeneration of a confocal family in $\R^{n+1}$; this gives rise to the so-called sphero-conical coordinates in $\R^{n+1}$.

The following lemma gives a spherical analog of classical facts concerning confocal families in the Euclidean space.

\begin{lemma}
\label{lem:FLambda}
Let $\lambda \in (-b, a_n)$.
The linear map $f_\lambda \colon \R^{n+1} \to \R^{n+1}$ given by a diagonal matrix
\[
f_\lambda = \diag \left( \sqrt{\frac{a_1-\lambda}{a_1}}, \ldots, \sqrt{\frac{a_n-\lambda}{a_n}}, \sqrt{\frac{b+\lambda}b} \right)
\]
has the following properties:
\begin{enumerate}
\item
It maps the spherical ellipsoid $E$ to the spherical ellipsoid $E_\lambda$;
\item
The points $x \in E$ and $f_\lambda(x) \in E_\lambda$ lie on the same $n-1$ spherical quadrics confocal to $E$;
\item
The pull-back by $f_\lambda$ of a homeoidal measure on $E_\lambda$ is a homeoidal measure on $E$.
\end{enumerate}
\end{lemma}
\begin{proof}
Since $q_\lambda = q \circ f_\lambda^{-1}$, the map $f_\lambda$ sends the cone $q^{-1}(0)$ to the cone $q_\lambda^{-1}(0)$. Besides
\[
\left.
\begin{matrix}
x_1^2 + \cdots + x_n^2 + x_0^2 = 1\\
\frac{x_1^2}{a_1} + \cdots + \frac{x_n^2}{a_n} - \frac{x_0^2}{b} = 0
\end{matrix}
\right\}
\Rightarrow \frac{a_1-\lambda}{a_1} x_1^2 + \cdots + \frac{a_n-\lambda}{a_n} x_n^2 + \frac{b+\lambda}b x_0^2 = 1,
\]
which implies that the image of $E$ is contained in the unit sphere. Thus $f_\lambda(E) = E_\lambda$.

For the second part it suffices to prove that
\[
\left.
\begin{matrix}
q(x) = 0\\
q_\mu(x) = 0
\end{matrix}
\right\}
\Rightarrow q_\mu(f_\lambda(x)) = 0.
\]
This follows from the linear relation
\[
\frac\lambda\mu q + \left( 1 - \frac\lambda\mu \right) q_\mu = q_\mu \circ f_\lambda,
\]
which can be checked by a direct computation.

Let $\omega_E$ and $\omega_{E_\lambda}$ denote the volume elements on $E$ and $E_\lambda$, respectively. By Lemma \ref{lem:HomDens}, $\frac{\omega_E}{\|\grad q\|}$ is a homeoidal measure on $E$, hence we need to show that
\begin{equation}
\label{eqn:Pullback}
f_\lambda^*\left(\frac{\omega_{E_\lambda}}{\|\grad q_\lambda\|}\right) = c \cdot \frac{\omega_E}{\|\grad q\|}
\end{equation}
for some constant $c$.
We have
\[
dr \wedge \frac{dq}{\|\grad q\|} \wedge \omega_E = \omega = dr \wedge \frac{dq_\lambda}{\|\grad q_\lambda\|} \wedge \omega_{E_\lambda},
\]
where $\omega$ denotes the volume element of $\R^{n+1}$. On the other hand, since $f_\lambda$ is a linear map, $f_\lambda^*(\omega) = c \cdot \omega$ for some constant $c$. Taking into account that $f_\lambda^*(dr) = dr$ when restricted to $E$ and $E_\lambda$, and that $f_\lambda^*(dq_\lambda) = dq$, we obtain equation \eqref{eqn:Pullback}.
\end{proof}

\begin{remark}
The second part of Lemma \ref{lem:FLambda} immediately implies the spherical Ivory lemma in some special cases. Take a parallelepiped with one vertex $x \in E$ and the opposite vertex $y \in E_\lambda$. Then one of the diagonals opposite to $xy$ is $f_\lambda(x) f_\lambda^{-1}(y)$. These two diagonals have equal lengths because
\[
\langle x, y \rangle = \langle f_\lambda(x), f_\lambda^{-1}(y) \rangle,
\]
due to the fact that  the operator $f_\lambda$ is self-adjoint. On $\Sph^2$, this proves the Ivory lemma in full generality. In higher dimensions, there are other diagonals opposite to $xy$. In order to prove that they have the same length one needs a generalization of Lemma \ref{lem:FLambda} to linear maps within subfamilies of confocal quadrics other than ellipsoids.
\end{remark}

In the hyperbolic case, the following modifications are needed. First, we need to diagonalize the quadratic form $q$ simultaneously with the Minkowski scalar product. The simultaneous diagonalization of indefinite quadratic forms is not possible in general (a simple example: $x^2 - y^2$ and $xy$). However, since $q$ defines an ellipsoid in the hyperbolic space, we can use the following lemma.
\begin{lemma}
Let $p$ and $q$ be two non-degenerate quadratic forms of index $1$ such that the light cone of $q$ lies inside the light cone of $p$. Then $p$ and $q$ can be simultaneously diagonalized.
\end{lemma}
\begin{proof}
Let $L_p$ and $L_q$ be the interiors of the light cones, that is, the sets of vectors whose $p$- and $q$-norms squared are non-positive. For $v \in L_p$, the orthogonal complement $v^{\perp_p}$ is disjoint with the interior of $L_p$, and likewise for $L_q$.  

Therefore we have a map
\[
L_p \to L_q, \quad v \mapsto (v^{\perp_p})^{\perp_q}
\]
whose  projectivization has a fixed point by Brouwer's theorem. Choose it as the first basis direction. The restrictions of $p$ and $q$ to the common orthogonal complement are positive definite, hence simultaneously diagonalizable.
\end{proof}

Due to the above lemma, we can assume that $q$ has the form \eqref{eqn:qDiag} with $a_i < b$ for all $i$ (which ensures that the light cone of $q$ lies within the light cone of the Minkowski scalar product). The associated quadratic forms are
\[
q_\lambda(x) = \frac{x_1^2}{a_1-\lambda} + \cdots + \frac{x_n^2}{a_n-\lambda} - \frac{x_0^2}{b-\lambda}, \quad a_1 > \lambda \ne a_i
\]
(their duals form a pencil spanned by the duals of $q$ and of the Minkowski scalar product). A hyperbolic analog of Lemma \ref{lem:FLambda} holds; the argument remains the same.

\begin{theorem} \label{LIconst}
A spherical or hyperbolic ellipsoid $E$ charged with a homeoidal density creates an electrostatic force field that vanishes inside $E$ (and $-E$, in the spherical case) and has the confocal ellipsoids $E_\lambda$ (and $-E_\lambda$, in the spherical case) as equipotential surfaces.
\end{theorem}
\begin{proof}
Lemma \ref{lem:EqualHeights} implies that the field inside an arbitrary homeoidal shell vanishes: the forces exerted on a point inside of a homeoid by two diametrally opposite truncated cones compensate each other, see Figure \ref{fig:NewtonToIvory} and the first part of the proof of Theorem \ref{thm:NewtonSph}. In the limit, as a homeoid around $E$ becomes infinitely thin, its electrostatic field tends to the field created by a homeoidal density on $E$ and thus vanishes inside $E$.

Take a point $x$ outside $E$. The potential created at a point $x$ by a homeoidal density on $E$ equals
\[
U(x) = \int_E u \, do,
\]
where $u(y) = u(\|x-y\|)$ is the potential of a point charge at $x$ described in Section \ref{sec:GravPot}, and $o$ is a homeoidal measure on $E$. 

\begin{figure}[ht]
\centering
\begin{picture}(0,0)%
\includegraphics{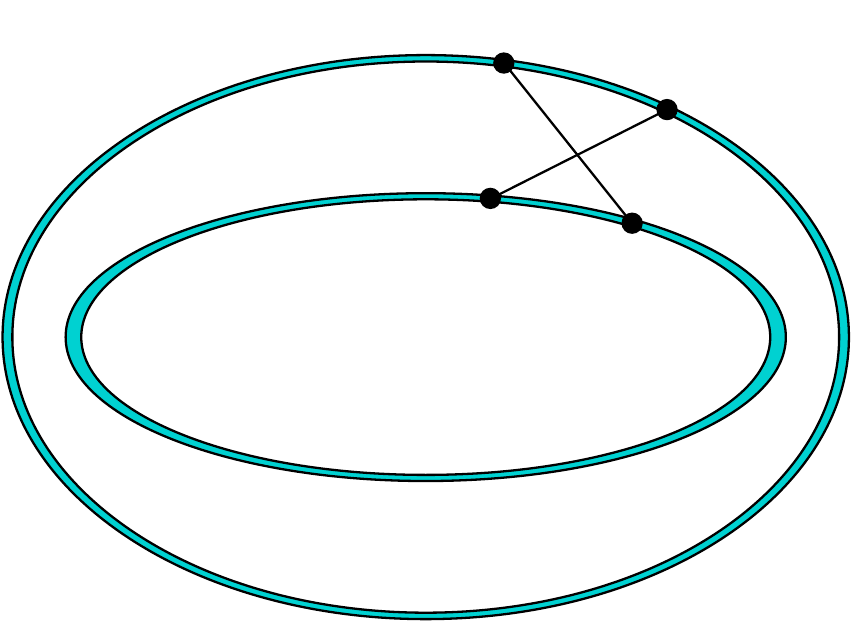}%
\end{picture}%
\setlength{\unitlength}{5801sp}%
\begingroup\makeatletter\ifx\SetFigFont\undefined%
\gdef\SetFigFont#1#2#3#4#5{%
  \reset@font\fontsize{#1}{#2pt}%
  \fontfamily{#3}\fontseries{#4}\fontshape{#5}%
  \selectfont}%
\fi\endgroup%
\begin{picture}(2780,2025)(-264,-540)
\put(1640,593){\makebox(0,0)[lb]{\smash{{\SetFigFont{10}{12.0}{\rmdefault}{\mddefault}{\updefault}{\color[rgb]{0,0,0}$f_\lambda^{-1}(x)$}%
}}}}
\put(883,661){\makebox(0,0)[lb]{\smash{{\SetFigFont{10}{12.0}{\rmdefault}{\mddefault}{\updefault}{\color[rgb]{0,0,0}$f_\lambda^{-1}(z)$}%
}}}}
\put(1960,1194){\makebox(0,0)[lb]{\smash{{\SetFigFont{10}{12.0}{\rmdefault}{\mddefault}{\updefault}{\color[rgb]{0,0,0}$x$}%
}}}}
\put(1324,1362){\makebox(0,0)[lb]{\smash{{\SetFigFont{10}{12.0}{\rmdefault}{\mddefault}{\updefault}{\color[rgb]{0,0,0}$z$}%
}}}}
\put(519, 57){\makebox(0,0)[lb]{\smash{{\SetFigFont{10}{12.0}{\rmdefault}{\mddefault}{\updefault}{\color[rgb]{0,0,0}$E$}%
}}}}
\put(498,-375){\makebox(0,0)[lb]{\smash{{\SetFigFont{10}{12.0}{\rmdefault}{\mddefault}{\updefault}{\color[rgb]{0,0,0}$E_\lambda$}%
}}}}
\end{picture}%
\caption{Proof of the Ivory theorem.}
\label{Iproof}
\end{figure}
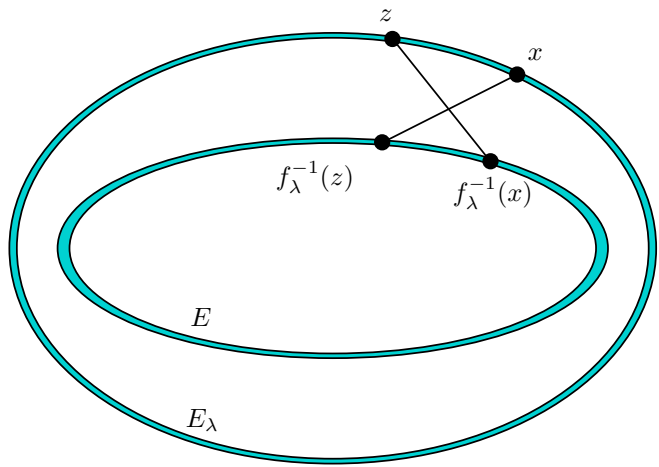

Let $E_\lambda$ be a confocal ellipsoid through $x$, see Figure \ref{Iproof}. Using the map $f_\lambda$ from Lemma \ref{lem:FLambda}, one transforms the above integral to an integral over $E_\lambda$:
\[
\int_E u\, do = \int_{E_\lambda} u \circ f_\lambda^{-1} \, do_\lambda,
\]
where $o_\lambda$ is a homeoidal measure on $E_\lambda$, due to the third part of Lemma \ref{lem:FLambda}. For every $z \in E_\lambda$, we have
\[
u \circ f_\lambda^{-1}(z) = u(\|x - f_\lambda^{-1}(z)\|) = u(\|f_\lambda^{-1}(x) - z\|).
\]
Indeed, by the second part of Lemma \ref{lem:FLambda}, the geodesic segments $x f_\lambda^{-1}(z)$ and $f_\lambda(x)z$ span the same coordinate parallelepiped; they are of the same length by the Ivory lemma. It follows that $u \circ f_\lambda^{-1}$ is the potential of a point charge at $f_\lambda^{-1}(x)$. Hence
\[
\int_{E_\lambda} u \circ f_\lambda^{-1} \, do_\lambda = U_\lambda(f_\lambda^{-1}(x))
\]
is the potential created at point $f_\lambda^{-1}(x)$ by the measure $o_\lambda$ on $E_\lambda$. But the point $f_\lambda^{-1}(x)$ lies inside the ellipsoid $E_\lambda$. Since by the first part of the theorem the potential of $o_\lambda$ is constant inside $E_\lambda$, the potential of the measure $o$ is the same at all points $x \in E_\lambda$.
\end{proof}

\subsection{Arnold's theorem}
\subsubsection{Hyperbolic surfaces}
\label{sec:HypSurf}
An algebraic surface $M$ of degree $d$ in $\R^n$ is called \emph{strictly hyperbolic} with respect to a point $x \in \R^n$ if $x \notin M$ and every line through $x$ intersects the projective closure of $M$ in $d$ distinct points.
If every line through $x$ intersects $M$ in $d$ not necessarily distinct points (but counting the algebraic multiplicities) then $M$ is called \emph{hyperbolic} with respect to $x$. The \emph{hyperbolicity domain} of $M$ is the union of points $x$ such that $M$ is (strictly) hyperbolic with respect to $x$. A surface is called (strictly) hyperbolic if its hyperbolicity domain is non-empty.

By a result of Nuij \cite{Nuij68}, the space of strictly hyperbolic surfaces is open and contractible. This implies \cite{HW07} that every strictly hyperbolic surface consists of $\lfloor \frac{d}{2} \rfloor$ nested projective ovaloids (that is hypersurfaces whose projective closure is isotopic to a sphere) and (for $d$ odd) one more component isotopic to $\R P^{n-1} \subset \R P^n$. The innermost ovaloid is projectively convex, and its interior is the hyperbolicity domain of the surface. See Figure \ref{fig:HypDeg3} for examples of hyperbolic curves of degree $3$.
Also by \cite{Nuij68}, every hyperbolic surface is a limit of strongly hyperbolic ones.

\begin{figure}[ht]
\begin{center}
\includegraphics[height=4cm]{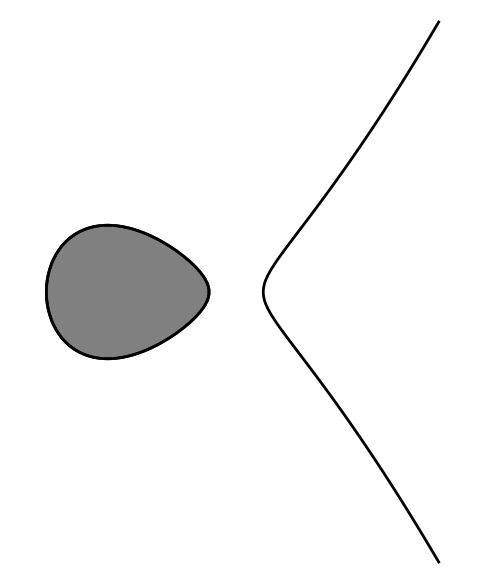} \hspace{.5cm} \includegraphics[height=4cm]{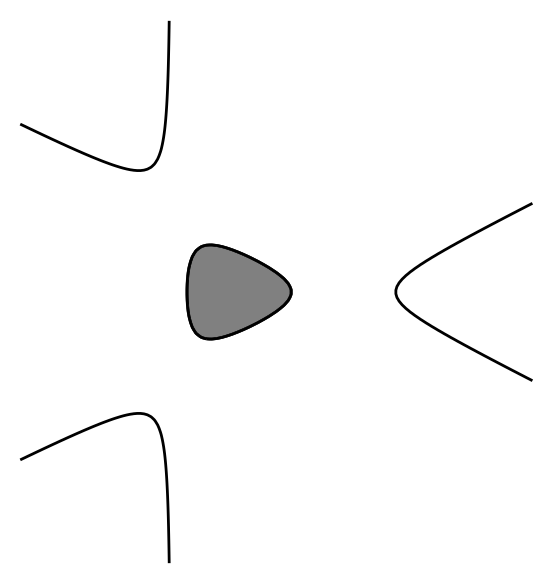} \hspace{.5cm} \includegraphics[height=4cm]{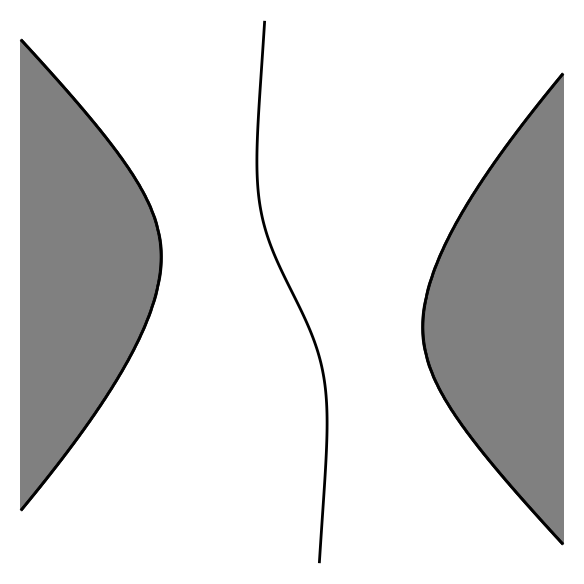}
\end{center}
\caption{Hyperbolic curves of degree $3$. The hyperbolicity domain is shaded.}
\label{fig:HypDeg3}
\end{figure}

\begin{example}
The curve $x^4 + y^4 = 1$ is not hyperbolic.

Ellipsoids and hyperboloids of two sheet are strictly hyperbolic; hyperboloids of one sheet are not.

The union of $d$ hyperplanes is hyperbolic, although not strictly. Its hyperbolicity domain is the complement to the hyperplanes.
\end{example}

Hyperbolic polynomials appeared in the works of Petrovsky \cite{Pet45} and G\aa{}rding \cite{Gar51} on partial differential equations. In the recent decades they found applications in various other domains of mathematics, see \cite{Pem12}.

\subsubsection{Arnold's theorem: Euclidean case}
\label{sec:Layer}
In the following we assume that the degree $d$ is at least $2$.
For a strictly hyperbolic surface $M = p^{-1}(0)$, a \emph{standard layer} is the shell
\[
\{x \in \R^n \mid 0 \le p(x) \le \epsilon\},
\]
where $\epsilon$ is small enough for the surface $p^{-1}(\epsilon)$ to be also strictly hyperbolic. Fix a point $x$ in the hyperbolicity domain of $M$. Charge a component of the standard layer positively if the corresponding component of $p^{-1}(0)$ lies ``closer'' to $x$ than the corresponding component of $p^{-1}(\epsilon)$, and negatively otherwise. For example, if $M$ is a hyperbola, then the layers along its branches get different signs. The intersection of the hyperbolicity domains of $p^{-1}(0)$ and $p^{-1}(\epsilon)$ will be called the hyperbolicity domain of the layer.

\begin{theorem}[Arnold \cite{Arn82}]
For every hyperbolic surface, a charged standard layer creates a zero electrostatic field in its hyperbolicity domain.
\end{theorem}

As in the case of homeoids and homeoidal charges, in the limit $\epsilon \to 0$ we obtain a hyperbolic surface charged with the density $\frac{1}{\|\grad p\|}$ (and with the sign of the charge  on different components subject to the same rule as above).
\begin{corollary}
A standard charge on a hyperbolic surface creates a zero electrostatic field in its hyperbolicity domain.
\end{corollary}

Exactly as in the Newton-Ivory situation (see Figure \ref{fig:NewtonToIvory}), Arnold's theorem follows from the lemma below.

\begin{lemma}
\label{lem:Arnold}
Let $p$ be a polynomial of degree $d \ge 2$ in one variable with $d$ roots $t_1, \ldots, t_d$, and let $t_1^\epsilon, \ldots, t_d^\epsilon$ be the roots of $p(t) - \epsilon$. Then
\[
\sum_{i=1}^d (t_i - t_i^\epsilon) = 0.
\]
\end{lemma}
\begin{proof}
%The following argument is simpler than the original one; its idea is used below in the hyperbolic and the spherical case.
By Vieta's formula,
\[
t_1 + \cdots + t_d = -\frac{a_{d-1}}{a_d} = t_1^\epsilon + \cdots + t_d^\epsilon,
\]
where $p(t) = a_d t^d + a_{d-1} t^{d-1} + \cdots$. The lemma follows.
\end{proof}

Arnold's theorem holds for non-strictly hyperbolic surfaces as well, since these can be approximated by strictly hyperbolic ones. The sign rule of the charge and the domain of the vanishing electrostatic field depend on the choice of approximation (there can be topologically different choices). For example, the coordinate axes in $\R^2$ charged by $\frac1y$, respectively $\frac1x$, create a zero electrostatic field in the domain $xy>0$.

V. Vassiliev and W. Ebeling \cite{Vas98} have shown that, outside of the 
hyperbolicity domain, the force field is algebraic if $d=2$ or $n=2$, and 
non-algebraic otherwise. See \cite{Vas02, KhL14} for surveys.

\subsubsection{Extension of Arnold's theorem to spaces of constant curvature}
An algebraic surface of degree $d$ in $\Sph^n \subset \R^{n+1}$ or $\HH^n \subset \R^{n,1}$ is the zero set of a homogeneous polynomial of degree $d$ in $n+1$ variables. Similarly to Section \ref{sec:HypSurf}, we call an algebraic surface $M$ hyperbolic with respect to a point $x$ if every great circle through $x$ intersects $M$ in $d$ distinct pairs of antipodal points (for $M \subset \Sph^n$) or if every geodesic through $x$ intersects $M$ in $d$ distinct points.

An algebraic surface $p^{-1}(0) \cap \Sph^n$ is hyperbolic if and only if the corresponding affine algebraic surface $p^{-1}(0) \cap \{x_0 = 1\}$ is hyperbolic. Topologically, a hyperbolic surface in $\Sph^n$ consists of $\lfloor \frac{d}2 \rfloor$ antipodal pairs of nested ovaloids and (for $d$ odd) one additional centrally symmetric ovaloid.

An algebraic surface in $p^{-1}(0) \cap \HH^n$ is hyperbolic if and only if the surface $p^{-1}(0) \cap \{x_0 = 1\}$ is hyperbolic and is contained in the Cayley-Klein ball $x_1^2 + \cdots + x_n^2 < 1$. (Otherwise through every point in $\HH^n$ there is a line that intersects the surface outside of the ball; hence the hyperbolicity domain is empty.) This implies that in the hyperbolic space there are no hyperbolic surfaces of odd degree (an odd degree affine hyperbolic hypersurface has a component isotopic to the projective hyperplane, hence must leave the Cayley-Klein ball). A hyperbolic surface of an even degree $d$ in $\HH^n$ consists of $d$ nested ovaloids.

A spherical or hyperbolic \emph{standard layer} is the shell between two level sets $\epsilon_1 \le p(x) \le \epsilon_2$ intersected with $\Sph^n$ or with $\HH^n$. Put on the components of a standard layer charges of a constant density and with plus or minus signs according to the rule described in Section \ref{sec:Layer}.

On $\Sph^n$, every standard layer is symmetric with respect to the center of the sphere. For surfaces of even degree, the charges at the opposite points have different signs, while for surfaces of odd degree they have the same sign, and therefore the resulting electrostatic field vanishes everywhere.

\begin{theorem} \label{Arnconst}
For every hyperbolic surface in $\Sph^n$ or $\HH^n$, a charged standard layer creates a zero electrostatic field in its hyperbolicity domain.
\end{theorem}
\begin{proof}
For surfaces of even degree this follows from the first part of Lemma \ref{lem:ArnSph} below.
In $\HH^n$ there are no hyperbolic surfaces of odd degree. In $\Sph^n$ a standard layer at a surface of odd degree creates a zero field everywhere for trivial reasons.
\end{proof}

An infinitesimally thin standard layer is equivalent to a surface charged with the density $\frac{1}{\|\grad p\|}$, where in the $\HH^n$ case the norm of the gradient is the Minkowski norm.

\begin{corollary}
A standard charge on a hyperbolic surface in $\Sph^n$ or $\HH^n$ creates a zero electrostatic field in its hyperbolicity domain.
\end{corollary}

\begin{lemma}
\label{lem:ArnSph}
Let $p$ be a homogeneous polynomial of degree $d$ in two variables. Assume that the cone $p^{-1}(0)$ intersects an open unit half-circle $\Sph^1_+$ (or, respectively, a branch $\HH^1$ of the hyperbola) in $d$ distinct points, and let $t_1, \ldots, t_d$ be the coordinates of these points in an arc-length parametrization of $\Sph^1$ (respectively, in a hyperbolic arc-length parametrization of $\HH^1$). For $\epsilon$ small enough, let $t_1^\epsilon, \ldots, t_d^\epsilon$ be the coordinates of the intersection of $\Sph^1_+$ (respectively $\HH^1$) with the curve $p^{-1}(\epsilon)$.
Then the following holds:
\begin{enumerate}
\item
If $d$ is even, then $\sum_{i=1}^d (t_i - t_i^\epsilon) = 0$;
\item
If $d$ is odd, then $\sum_{i=1}^d \left(t_i - \frac{t_i^{-\epsilon} + t_i^\epsilon}2\right) = 0$.
\end{enumerate}
\end{lemma}
\begin{proof}
Let us discuss the hyperbolic case first. Choose coordinates $x,y$ in $\R^2$ so that $\HH^1 = \{(x,y) \mid xy = 1, x > 0\}$. Let $(x_i, y_i)$ be the coordinates of the $i$-th intersection point. We have
\[
p(x_i,y_i) = 0, \quad x_iy_i = 1.
\]
Consider the degree $2d$ polynomial
\[
P(x) = x^d p\left(x, \frac1x\right).
\]
By construction, $x_1, \ldots, x_d$ are roots of $P$. Since $P$ contains only even degree monomials, its other $d$ roots are $-x_1, \ldots, -x_d$. Hence we have
\[
x_1 \cdot \ldots \cdot x_d = \sqrt{(-1)^d\frac{a_{0}}{a_{2d}}},
\]
where $a_0$ is the constant term of $P$, and $a_{2d}$ is the leading coefficient.

Let $d$ be even. Then, similarly to the above, $\pm x^\epsilon_1, \ldots, \pm x^\epsilon_d$ are the roots of the polynomial $P(x) - \epsilon x^d$, and we have
\[
x^\epsilon_1 \cdot \ldots \cdot x^\epsilon_d = \sqrt{(-1)^d\frac{a_{0}}{a_{2d}}} = x_1 \cdot \ldots \cdot x_d.
\]
A hyperbolic arc length parmetrization of $\HH^1$ is given by $t = \log x$. Hence we have
\[
x^\epsilon_1 \cdot \ldots \cdot x^\epsilon_d = x_1 \cdot \ldots \cdot x_d \Rightarrow \sum_{i=1}^d (t_i - t_i^\epsilon) = 0.
\]

Now let $d$ be odd. Then the $d$ intersection points of the hyperbola with $p^{-1}(\epsilon)$ correspond to $d$ roots $x^\epsilon_1, \ldots, x^\epsilon_n$ of the polynomial $P(x) - \epsilon x^d$. This polynomial contains a monomial of odd degree $x^d$. Due to
\[
P(-x) - \epsilon (-x)^d = P(x) + \epsilon x^d,
\]
its other $d$ roots are $-x^{-\epsilon}_1, \ldots, -x^{-\epsilon}_n$. It follows that
\[
x^\epsilon_1 \cdot \ldots \cdot x^\epsilon_n \cdot x^{-\epsilon}_1 \cdot \ldots \cdot x^{-\epsilon}_n = -\frac{a_{0}}{a_{2d}} = (x_1 \cdot \ldots \cdot x_{d})^2,
\]
and hence
\[
\sum_{i=1}^d \left(t_i - \frac{t_i^{-\epsilon} + t_i^\epsilon}2\right) = \sum_{i=1}^d \left(\log(x_i) - \log\sqrt{x^\epsilon_i x^{-\epsilon}_i}\right) = 0.
\]

In the spherical case we change the Euclidean coordinates $(x,y)$, to 
$$
u = x+iy,\ v = x-iy,
$$
so that $x^2 + y^2 = uv$. The rest of the proof is the same, with $u$ substituted for $x$.
\end{proof}

\subsection{A final remark}
The Ivory lemma holds on ellipsoids and, in the full generality, in Riemannian manifolds with St\"ackel nets. Is there an analog of the Ivory theorem for St\"ackel nets? Building on the fact that the gravitational potential is a harmonic function, one could conjecture that every St\"ackel net is compatible with a harmonic coordinate system. This seems to be false in general, but is true in a special case. With the help of the ellipsoidal coordinates, one can construct harmonic functions on an ellipsoid whose level sets are the intersections of the ellipsoid with confocal quadrics. However, there is no evident relation with the Ivory lemma: the proof of the Ivory theorem (Figure \ref{Iproof}) relies on the fact that the potential of a point is rotationally symmetric, and there is no rotational symmetry on the ellipsoid.

\end{document}